\documentclass[a4paper, twoside]{scrartcl}
\usepackage{etex}
\usepackage[a4paper, top=88pt, bottom=88pt, left=72pt, right=72pt, headsep=16pt, footskip=28pt]{geometry}
\usepackage{hyperref}
\hypersetup{
	breaklinks=true,
	colorlinks=true,
	linkcolor=blue,
	citecolor=blue,
	urlcolor=blue,
}
\usepackage{amsfonts, amssymb, amsmath, amsthm, amstext, amssymb, mathtools}
\usepackage{caption, subcaption, float}
\usepackage{color,xcolor,graphicx}
\usepackage{enumerate, enumitem, moreenum}
\usepackage[authoryear,sort,round]{natbib}  
\newcommand*{\arXiv}[1]{\bgroup\color{blue}\href{https://arxiv.org/abs/#1}{arXiv:#1}\egroup}
\newcommand*{\doi}[1]{\bgroup\color{blue}\href{https://doi.org/#1}{doi:#1}\egroup}
\newcommand*{\email}[1]{\bgroup\color{blue}\href{mailto:#1}{#1}\egroup}
\renewcommand*{\url}[1]{\bgroup\color{blue}\href{#1}{#1}\egroup}
\usepackage{enumitem, moreenum}
\setlist[enumerate]{nosep}
\setlist[itemize]{nosep}
\usepackage{mleftright} \mleftright
\renewcommand{\qedsymbol}{$\blacksquare$}
\renewenvironment{proof}[1][\proofname]{\noindent{\bfseries\sffamily #1.} }{\hfill\qedsymbol\medskip}
\usepackage[labelfont={sf,bf}]{caption}
\DeclareCaptionLabelSeparator{figlabelsep}{\,\,\,}
\usepackage{scrlayer-scrpage, xhfill}
\automark[section]{section}
\setkomafont{pageheadfoot}{\normalcolor\sffamily}
\setkomafont{pagenumber}{\normalfont\normalsize\sffamily}
\clearpairofpagestyles
\let\oldtitle\title
\renewcommand{\title}[1]{\oldtitle{#1}\newcommand{\theshorttitle}{#1}}
\newcommand{\shorttitle}[1]{\renewcommand{\theshorttitle}{#1}}
\let\oldauthor\author
\renewcommand{\author}[1]{\oldauthor{#1}\newcommand{\theshortauthor}{#1}}
\newcommand{\shortauthor}[1]{\renewcommand{\theshortauthor}{#1}}
\cohead{\xrfill[0.525ex]{0.6pt}~\theshorttitle~\xrfill[0.525ex]{0.6pt}}
\cehead{\xrfill[0.525ex]{0.6pt}~\theshortauthor~\xrfill[0.525ex]{0.6pt}}
\newcommand{\theabstract}[1]{\par\bgroup\noindent\textbf{\textsf{Abstract.}} #1\egroup}
\newcommand{\thekeywords}[1]{\par\smallskip\bgroup\noindent\textbf{\textsf{Keywords.}}\newcommand{\and}{ $\bullet$ } #1\egroup}
\newcommand{\themsc}[1]{\par\smallskip\bgroup\noindent\textbf{\textsf{2020 Mathematics Subject Classification.}}\newcommand{\and}{ $\bullet$ } #1\egroup}
\newcommand*{\affilref}[1]{\ref{affiliation#1}}
\newcommand*{\affiliation}[3]{
	\footnotetext[#1]{\label{affiliation#2} #3}
}
\addtokomafont{disposition}{\sffamily}

\usepackage{dsfont}  
\usepackage[normalem]{ulem} 
\usepackage{nicefrac}
\usepackage{algpseudocode,algorithm}

\usepackage[noabbrev,capitalise,nosort,nameinlink]{cleveref}

\crefname{counterexample}{Counterexample}{Counterexamples}
\Crefname{counterexample}{Counterexample}{Counterexamples}

\usepackage{adjustbox}
\usepackage{tikz}
\usetikzlibrary{bayesnet,er,trees,shapes.symbols,mindmap,arrows,snakes,shapes.misc,shapes.arrows,chains,matrix,positioning,scopes,decorations.pathmorphing,patterns}
\usepackage{tikz-cd} 

\usetikzlibrary{decorations.markings}
\tikzset{
  negate/.style={
    decoration={
      markings,
      mark= at position 0.5 with {
        \node[transform shape] (tempnode) {$/$};
      },
    },
    postaction={decorate},
  },
}

\usepackage{chngcntr}

\usepackage{thmtools,thm-restate}

\usepackage{array}
\newcolumntype{x}[1]{>{\centering\arraybackslash\hspace{0pt}}p{#1}}
\usepackage{booktabs,colortbl,xcolor,xfrac}
\usepackage{multirow}

\usepackage{ragged2e}

\usepackage{matlab-prettifier}

\newcommand*{\betterpoints}{higher-order points\xspace}
\newcommand*{\trv}{\bP_{\textup{tar}}}
\newcommand*{\simpv}{\bP_{\textup{ref}}}
\newcommand*{\trd}{\rho_{\textup{tar}}}
\newcommand*{\simpd}{\rho_{\textup{ref}}}

\newcommand*{\iidsim}{\stackrel{\textsc{iid}}{\sim}}

\newcommand*{\cA}{\mathcal{A}}
\newcommand*{\cC}{\mathcal{C}}
\newcommand*{\cF}{\mathcal{F}}

\newcommand*{\cX}{\mathcal{X}}
\newcommand*{\bPU}{\bP_{\mathcal{U}}}
\newcommand*{\bPN}{\bP_{\mathcal{N}}}

\DeclareMathOperator{\erf}{erf}
\DeclareMathOperator{\sgn}{sgn}
\DeclareMathOperator{\diver}{div}
\DeclareMathOperator{\GL}{GL}
\DeclareMathOperator{\Unif}{Unif}
\DeclareMathOperator{\BVHK}{BVHK}

\newcommand{\grayrule}{\arrayrulecolor{black!10}\specialrule{0.2pt}{2\jot}{2\jot}\arrayrulecolor{black}}

\newcommand*{\bR}{\mathbb{R}}
\newcommand*{\bN}{\mathbb{N}}

\newcommand*{\bE}{\mathbb{E}}

\newcommand*{\bP}{\mathbb{P}}

\newcommand*{\cN}{\mathcal{N}}
\newcommand*{\cO}{\mathcal{O}}

\newcommand*{\fN}{\mathfrak{N}}
\newcommand*{\fw}{\mathfrak{w}}
\newcommand*{\fX}{\mathfrak{X}}
\newcommand*{\fY}{\mathfrak{Y}}

\newcommand*{\quark}{\setbox0\hbox{$x$}\hbox to\wd0{\hss$\cdot$\hss}}
\newcommand*{\rd}{\mathrm{d}}

\DeclareMathOperator{\Id}{Id}

\DeclarePairedDelimiter\abs{\lvert}{\rvert}%
\DeclarePairedDelimiter\norm{\lVert}{\rVert}%
\makeatletter
\let\oldabs\abs
\def\abs{\@ifstar{\oldabs}{\oldabs*}}
\let\oldnorm\norm
\def\norm{\@ifstar{\oldnorm}{\oldnorm*}}
\makeatother

\newcommand*{\set}[2]{\{ #1 \mid #2 \}}

\definecolor{darkgreen}{rgb}{0,0.4,0}

\newenvironment{greennote}{\par\color{darkgreen}}{\par}

\newcommand*{\defeq}{\coloneqq}
\renewcommand{\#}{\ensuremath \sharp}
\renewcommand{\ge}{\geqslant}
\renewcommand{\geq}{\geqslant}
\renewcommand{\le}{\leqslant}
\renewcommand{\leq}{\leqslant}

\theoremstyle{plain}
\newtheorem{theorem}{\sffamily Theorem}[section]
\newtheorem{proposition}[theorem]{\sffamily Proposition}
\newtheorem{lemma}[theorem]{\sffamily Lemma}
\newtheorem{corollary}[theorem]{\sffamily Corollary}
\theoremstyle{definition}
\newtheorem{definition}[theorem]{\sffamily Definition}

\newtheorem{example}[theorem]{\sffamily Example}

\newtheorem{remark}[theorem]{\sffamily Remark}

\newtheorem{terminology}[theorem]{\sffamily Terminology}

\numberwithin{equation}{section}
\numberwithin{figure}{section}
\numberwithin{table}{section}

\usepackage{acro}
\DeclareAcronym{BIP}{short=BIP, long=Bayesian inverse problem}
\DeclareAcronym{IID}{short=IID, long=independent and identically distributed}
\DeclareAcronym{MC}{short=MC, long=Monte Carlo}
\DeclareAcronym{LAIS}{short=LAIS, long=layered adaptive importance sampling}
\DeclareAcronym{MCMC}{short=MCMC, long=Markov chain Monte Carlo}
\DeclareAcronym{ODE}{short=ODE, long=ordinary differential equation}
\DeclareAcronym{PDE}{short=PDE, long=partial differential equation}
\DeclareAcronym{QMC}{short=QMC, long=quasi-Monte Carlo}
\DeclareAcronym{RKHS}{short=RKHS, long=reproducing kernel Hilbert space}
\DeclareAcronym{SG}{short=SG, long=sparse grids}
\DeclareAcronym{SVGD}{short=SVGD, long=Stein variational gradient descent}
\DeclareAcronym{TQMC}{short=TQMC, long=transported QMC}
\DeclareAcronym{TSG}{short=TSG, long=transported SG}

\title{Transporting Higher-Order Quadrature Rules}
\subtitle{Quasi-Monte Carlo Points and Sparse Grids for Mixture Distributions}
\shorttitle{Transporting Higher-Order Quadrature Rules: QMC and Sparse Grids}
\author{%
	Ilja~Klebanov\textsuperscript{\affilref{FUB}}%
	\and
	T.~J.~Sullivan\textsuperscript{\affilref{Warwick},\affilref{Turing}}%
}
\shortauthor{I.~Klebanov and T.~J.~Sullivan}

\begin{document}
	\maketitle
	
	\affiliation{1}{FUB}{Freie Universit\"at Berlin, Arnimallee 6, 14195 Berlin, Germany (\email{klebanov@zedat.fu-berlin.de})}
	\affiliation{2}{Warwick}{Mathematics Institute and School of Engineering, University of Warwick, Coventry, CV4 7AL, United Kingdom (\email{t.j.sullivan@warwick.ac.uk})}
	\affiliation{3}{Turing}{Alan Turing Institute, 96 Euston Road, London, NW1 2DB, United Kingdom}
	
	\begin{abstract}\small
		\theabstract{Integration against, and hence sampling from, high-dimensional probability distributions is of essential importance in many application areas and has been an active research area for decades.
One approach that has drawn increasing attention in recent years has been the generation of samples from a target distribution $\trv$ using transport maps:
if $\trv = T_\# \simpv$ is the pushforward of an easily-sampled probability distribution $\simpv$ under the transport map $T$, then the application of $T$ to $\simpv$-distributed samples yields $\trv$-distributed samples.
This paper proposes the application of transport maps not just to random samples, but also to quasi-Monte Carlo points, higher-order nets, and sparse grids in order for the transformed samples to inherit the original convergence rates that are often better than $N^{-1/2}$, $N$ being the number of samples/quadrature nodes.
Our main result is the derivation of an explicit transport map for the case that $\trv$ is a mixture of simple distributions, e.g.\ a Gaussian mixture, in which case application of the transport map $T$ requires the solution of an \emph{explicit} ODE with \emph{closed-form} right-hand side.
Mixture distributions are of particular applicability and interest since many methods proceed by first approximating $\trv$ by a mixture and then sampling from that mixture (often using importance reweighting).
Hence, this paper allows for the sampling step to provide a better convergence rate than $N^{-1/2}$ for all such methods.
}
		\thekeywords{Quasi-Monte Carlo%
\and
sampling of probability distributions%
\and
sparse grids%
\and
transport maps%
}
		\themsc{62D99
\and%
65C05
\and%
65D32
\and%
65D40
\and%
11K36 
}
	\end{abstract}

\section{Introduction}
\label{section:Intro}

When estimating the expected value $\bE_{\trv}[f] \defeq \int f(y)\, \rd \trv(y)$ of some function $f\in L^{1}(\trv)$ with respect to some target probability distribution $\trv$, which is a crucial task in many areas of applied mathematics and statistics such as Bayesian inference, one often faces one or more of the following challenges:
\begin{itemize}
	\item
	high dimension: for a given desired level of accuracy, the number of grid points required by classical quadrature rules grows exponentially with the dimension (`curse of dimensionality'), rendering such methods infeasible;
	\item
	multimodal distributions $\trv$ that cannot be well approximated by more simple distributions such as a normal distribution;
	\item
	the probability density of $\trv$ is only given up to an unknown normalisation constant.
\end{itemize}
By far the most widely used methodologies that can tackle these issues are \ac{MC} and \ac{MCMC} methods, which approximate the expected value by an empirical mean or ergodic average
\begin{equation}
	\label{MCapproximation}
	\bE_{\trv}[f]
	\approx
	\frac{1}{N} \sum_{n=1}^{N} f(Y_n) .
\end{equation}
In the case of `vanilla' \ac{MC}, the points $Y_{1}, \dots, Y_{N}$ are \ac{IID} with law $\trv$, if $\trv$ is simple enough to draw direct samples from it;
in the case of \ac{MCMC}, $Y_{1}, \dots, Y_{N}$ form a Markov chain with asymptotic distribution $\trv$ as $N \to \infty$.
In both cases, the central limit theorem (CLT) and ergodicity arguments guarantee that the approximation error $\abs{\bE_{\trv}[f] - \frac{1}{N} \sum_{n=1}^{N} f(Y_n)}$ has a convergence rate of $N^{-1/2}$ (see e.g.\ \citet[Chapter~17]{meyn2009markov}), independently of the dimension $d$ --- although, via the variance of the integrand $f$, $d$ may still appear implicitly as a premultiplier of the convergence rate.
While this breaks the curse of dimensionality, the convergence is still rather slow, particularly in applications where $f$ is costly to evaluate.
Under suitable regularity assumptions on $f$, there exist several well-known alternatives to \ac{MC} methods with better convergence rates:
\begin{itemize}
	\item \ac{QMC} methods \citep{niederreiter1992random,fang1993number,caflisch1998monte,dick2013high};
	\item higher-order digital nets \citep{dick2010digital};
	\item \acl{SG} (\acs{SG}; \citealp{smolyak1963quadrature,zenger1991sparse,gerstner1998numerical}).
\end{itemize}
We will refer to these types of point sequences as \emph{\betterpoints}.
However, such point sequences have only been constructed for very few, simple distributions $\trv$ such as uniform distributions on the cube, $\bPU = \Unif([0,1]^d)$, and (standard) normal distributions, $\bPN = \cN(0,\Id_d)$.
In this paper, we suggest constructing \emph{\betterpoints} for other distributions by the application of transport maps, an approach already suggested for \ac{MC} samples as an alternative to \ac{MCMC} \citep{el2012bayesian,marzouk2016introduction,parno2018transport}:
If a sequence $X_1, \dots, X_N$ with the desired convergence rate can be constructed for a simple distribution $\simpv$ and $T$ is a measurable map such that $T_\# \simpv = \trv$, then the sequence given by $Y_n \defeq T(X_n)$ will inherit this convergence rate for the distribution $\trv$.
The reason for this is the change of variables formula,
\begin{equation}
	\label{equ:ChangeOfVariablesFormula}
	\bE_{\trv}[f(Y)]
	=
	\bE_{\simpv}[f\circ T],
	\qquad
	\frac{1}{N} \sum_{n=1}^{N} f(Y_n)
	=
	\frac{1}{N} \sum_{n=1}^{N} f\circ T(X_n),
\end{equation}
which provides a \emph{dual viewpoint} on the transport idea:
Integrating the function $f$ against a complicated distribution $\trv$, which we try to rewrite as the pushforward of a simple distribution $\simpv$ under a transport map $T$, is equivalent to integrating the transformed function $f\circ T$ against the simple distribution $\simpv$
(and similarly, applying the function $f$ to transformed points $Y_{n} = T(X_{n})$ is equivalent to applying the transformed function $f\circ T$ to the original points $Y_{n}$, which really is a tautology), cf.\ \Cref{fig:transport_map_to_Gaussian_mixture_grid_and_pulled_back_function}.
This changes the task of sampling the target distribution into the task of constructing a suitable transport map, which is typically extremely challenging, especially in high dimensions.
A further note of caution is called for:
The regularity assumptions needed for establishing the convergence rate of quadrature using the points $X_{n}$ now have to be verified for the function $f\circ T$ instead of $f$, which may modify the class of functions $\cC$ to which such methods are applicable;
cf.\ \Cref{section:TransportMaps}).
This paper does not attempt to give a comprehensive answer to this issue.
However, in most practical applications, there is no reason to presume that it is more likely for $\cC$ to contain $f$ rather than $f \circ T$ (unless one is interested in very specific quantities of interest, such as the mean or (co-)variance of $\trv$, in which case $f$ is a polynomial).

This paper's approach to the construction of transport maps is based on \acp{ODE} and the corresponding continuity equations:
The initial $\simpv$-distributed points are transported to $\trv$-distributed points by a flow map $\Phi_{t}$ of an \ac{ODE} with appropriately chosen right-hand side, which guarantees that the final flow map at $t=1$ defines a transport map from $\simpv$ to $\trv$.
While this idea has a long history \citep{moser1965on,dacorogna1990on} with several recent implementations \citep{heng2015gibbs, reich2011dynamical, reich2019schroedinger, liu2017stein}, to the best of our knowledge it has not yet been applied to \betterpoints.

The main contribution of this paper is the derivation of an \ac{ODE} with \emph{analytic} right-hand side for the case that $\trv$ is a mixture of simple distributions such that its flow map $\Phi_{t}$ for $t=1$ is an exact transport map from $\simpv$ to $\trv$.
This case should by no means be seen as a toy example.
Due to their flexibility, mixture modes have proven themselves extremely useful in a variety of applications in e.g.\ machine learning and data analysis \citep{Bishop2006pattern,McLachlan2019mixture}.
While most practical distributions (e.g.\ Bayesian posteriors) are not mixtures, many methodologies concentrate on approximating them by mixture distributions.
Such methods include
\begin{itemize}
	\item
	variational inference \citep{blei2017variational};
	\item 
	all variants of \acl{LAIS} (\acs{LAIS}; \citealp{martino2017layered,bugallo2017adaptive,martino2017anti});
	\item 
	a variety of kernel methods including kernel herding \citep{chen2010super,lacoste2015sequential}.
\end{itemize}
Subsequently, expected values with respect to the original distribution can be estimated by empirical means of samples from the mixture, possibly using importance reweighting (\citealp[Section~5.7]{rubinstein2016simulation}; \citealp[Section~3.3]{robert2004monte}) in order to correct for the approximation error.
Again, mixture distributions are favorable as importance sampling distributions due to their flexibility to approximate e.g.\ multimodal distributions.
It appears lucrative to perform the second step using (transported) \betterpoints in place of random samples in order to obtain better convergence rates.
We demonstrate this procedure by performing \ac{LAIS} with transported \ac{QMC} points in \Cref{section:Lais_with_QMC}.
Furthermore, this approach extends naturally to copulas of mixture distributions \citep{gunawan2021flexible}.

\begin{remark}
	\label{remark:QMC_points_from_each_component_separately}
	One might object that an easier approach to obtain the desired convergence rate is to apply \ac{QMC} or \ac{SG} to each mixture component \emph{separately} as sketched in \Cref{alg:componentwise_transport_sampling_of_mixtures} below.
	In fact, such methods are the topic of current research \citep{Cui2023quasimonte} and work well for small and moderate numbers of mixture components.
	Note, however, that
	the affordable budget $N$ of quadrature points might be comparable or even smaller than the number of mixture components $J$, in which case the convergence rate of the above approach appears pointless, cf.\ \Cref{fig:MC_QMC_error_over_J_dim_2}.
\end{remark}

The paper is structured as follows.
After commenting on related work in \Cref{section:RelatedWork}, \Cref{section:TransportMaps} discusses the transportation of point sets for numerical quadrature and reviews how the convergence rates of such quadrature schemes are assessed.

After giving some brief preliminaries on the construction of transport maps in \Cref{section:ConstructionTransportMaps}, in \Cref{section:TransportToMixtures} we discuss how \betterpoints from a mixture of simple distributions can be produced by means of transport.
In particular, we introduce and compare the performance of transported \ac{MC}, \ac{TQMC}, and \ac{TSG}.
Using the example of \ac{LAIS}, we show how this procedure can be combined with several state-of-the-art inference methods to yield practically relevant methodologies.

\Cref{section:NumericalExperiments} illustrates our approach on a suite of numerical test cases with some implementation details given in \Cref{section:ImplementationDetails}.
We test both \ac{TQMC} on its own, and \ac{TQMC} within \ac{LAIS}, against an extensive suite of integrands with dimensions $d$ varying from $2$ to $50$.

\Cref{section:Conclusion} gives some closing remarks, and some technical supporting results can be found in \Cref{section:Appendix}.
	
\section{Related Work}
\label{section:RelatedWork}

Prior to this work, surprisingly few papers have addressed the application of transport maps to \ac{QMC} sequences.
\citet{basu2016transformations} considered such transport maps in the case where both $\simpv$ and $\trv$ are uniform distributions on different but fairly simple domains, such as a simplex or a sphere.
Their focus lies on building low discrepancy sequences on domains $\cX$ different from the unit cube, not to represent given probability distributions in $\bR^d$.

\citet{kuo2010lattice} and \citet{nichols2014lattice} analyse a transport map from the uniform distribution on a cube to a product measure on $\bR^{d}$, such as the standard normal distribution, cf.\ \Cref{example:TransportMapUniformToNormal}.
The intricate analysis of the class of functions to which such transported \ac{QMC} sequences are applicable, which only holds true for this rather simple class of distributions (in particular, due to the product structure, the corresponding transport map acts in a coordinatewise manner), shows how difficult this step turns out to be.
An algorithm for the transport of QMC points to mixture distributions has recently been suggested by \citet{Cui2023quasimonte}.
While their approach to treat each mixture component separately works for a small number $J$ of mixture components (compared to the number $N$ of evaluation points), it is likely to fail when $J$ is of the order of (or even larger than) $N$, in which case a genuine (non-componentwise) transport map to the mixture appears necessary.

To the best of our knowledge, the transport of sparse grids to other probability distributions has been discussed only in special cases in dimension $d=3$ by \citet{rodriguez2008numerical} for the approximation of certain integrals arising in computational chemistry.

The idea of using flow maps of suitably constructed \ac{ODE}s as transport maps in order to sample from some target distribution is not new.
\citet{reich2011dynamical}, \cite{daum2012particle} and \citet{heng2015gibbs} suggested several approaches that are typically summarised as \emph{particle flow methods}.
Further algorithms can be viewed as time-discretised versions of such particle flows, moving from a tractable distribution to some target distribution via a sequence of intermediate distributions.
These include \emph{sequential Monte Carlo methods} \citep{delmoral2004feynman,delmoral2006smc}, \emph{annealed importance sampling} \citep{neal2001annealed}, and so-called \emph{particle-based variational inference} \citep{liu2019understanding,wang2019accelerated,chen2018unified,chen2018stein,chen2019stein}, with Stein variational gradient descent (SVGD; \citealt{liu2016stein,liu2017stein,duncan2019stein}) as its arguably most prominent method.
However, most of particle flow methods involve some sort of interaction between the particles, which might break the delicate structure of \betterpoints and eliminate the desired gain in convergence rate.
This might be the reason why such methods have not yet been applied to \betterpoints.
Note that the ODE we construct in this paper involves no particle interactions (in fact, each point can be transported separately), hence such issues do not arise in our case.
	\section{Transporting Points for Numerical Quadrature}
\label{section:TransportMaps}

\begin{definition}
	Let $(\Omega_0,\cA_0)$ and $(\Omega_1,\cA_1)$ be measurable spaces.
	The \emph{pushforward} of a probability measure $\bP_{0}$ on $(\Omega_0,\cA_0)$ through a measurable map $T \colon \Omega_0 \to \Omega_1$ is the probability measure $T_{\#} \bP_{0}$ on $(\Omega_{1}, \cA_{1})$ defined by
	\begin{equation}
		\label{eq:DeterministicPushforward}
		(T_{\#} \bP_{0}) (B) \defeq \bP_{0} ( T^{-1} (B) )
		\quad
		\text{for each $B \in \cA_{1}$.}
	\end{equation}
	We also call the map $T$ a \emph{transport map} from $\bP_0$ to $\bP_1$ when $\bP_1$ is the pushforward of $\bP_0$ under $T$, i.e.\ when $T_\# \bP_0 = \bP_1$ or $T_\# \colon \bP_0 \mapsto \bP_1$.
\end{definition}

Whenever $\Omega_{0}$ (resp.\ $\Omega_{1}$) is a Borel-measurable subset of $\bR^{d}$, the $\sigma$-algebra $\cA_{0}$ (resp.\ $\cA_{1}$) is assumed to be the corresponding Borel $\sigma$-algebra.

\begin{remark}
	\label{remark:ConsequenceOfTransportMapsForSampling}
	As a consequence, if $X$ is a $\bP_0$-distributed random variable, $X\sim\bP_0$, and $T$ is a transport map from $\bP_{0}$ to $\bP_{1}$, then $Y \defeq T(X)$ is $\bP_1$-distributed, $Y\sim\bP_1$.
\end{remark}

\begin{remark}
	\label{remark:StochasticTransport}
	This paper focusses on \emph{deterministic} transport maps $T$ from $\bP_0$ to $\bP_1$.
	There is also significant interest in \emph{stochastic} transports, i.e.\ in transition kernels $\kappa \colon \Omega_{0} \times \cA_{1} \to [0, 1]$ such that
	\begin{equation}
		\label{eq:StochasticPushforward}
		(\kappa_{\#} \bP_{0}) (B) \defeq \int_{\Omega_{0}} \int_{B} \kappa (\omega_{0}, \rd \omega_{1}) \, \bP_{0} (\rd \omega_{0})
		\quad
		\text{for each $B \in \cA_{1}$}
	\end{equation}
	satisfies $\kappa_{\#} \bP_{0} = \bP_{1}$.
	(The deterministic case is the special case $\kappa (\omega_{0}, \rd \omega_{1}) = \delta_{T(\omega_{0})} (\rd \omega_{1})$, where $\delta_{a}$ denotes the unit Dirac measure centred at $a$.)
	\citet{reich2019schroedinger} offers a review of such methods in the context of Bayesian data assimilation.
\end{remark}

If $\Omega_0,\Omega_1\subseteq\bR^d$ are Borel subspaces and $d=1$, then a transport map $T$ with $T_\# \colon \bP_0 \mapsto \bP_1$ can be defined using the cumulative distribution functions $F_0$ of $\bP_0$ and $F_1$ of $\bP_1$ via $T \defeq F_1^{-1}\circ F_0$, where $F_1^{-1}(x) \defeq \inf \set{ y\in\bR }{ F_1(y)\geq x }$, an approach known as the inverse transform method \citep[Chapter~2.3.1]{rubinstein2016simulation}.
In dimension $d>1$, a similar construction can be performed which is called the Knothe--Rosenblatt rearrangement \citep[Chapter~1]{villani2008optimal}, but it becomes computationally infeasible in high dimension except for certain special cases.
One such case occurs when $\simpv$ and $\trv$ factorise (that is, they are product measures of one-dimensional distributions) and the inverse transform method can be applied in each dimension separately, as in the following classical example:

\begin{example}
	\label{example:TransportMapUniformToNormal}
	Let $\bP_0 \defeq \Unif((0,1)^d)$ and $\bP_1 \defeq \cN(0,\Id_d)$.
	The transformation $T^\cN\colon (0,1)^d \to \bR^d$ given by
	\begin{equation}
		\label{equ:TransportMapUniformToNormal}
		T^\cN(x) \defeq \big(\sqrt{2}\, \erf^{-1}\left(2 x_j - 1\right) \big)_{j=1,\dots,d}
	\end{equation}
	is a transport map from $\bP_0$ to $\bP_1$, where $\erf (z) \defeq \frac{2}{\sqrt{\pi}} \int_{0}^{z} e^{-t^{2}} \, \rd t$ is the error function and $\erf^{-1}$ is its inverse.
	The action of $T^{\cN}$ on an equidistant grid is illustrated in \Cref{figure:TransportMapUniformToNormal}.
\end{example}

\begin{figure}[t]
	\centering
	\begin{subfigure}[b]{0.35\textwidth}
		\centering
		\includegraphics[width=\textwidth]{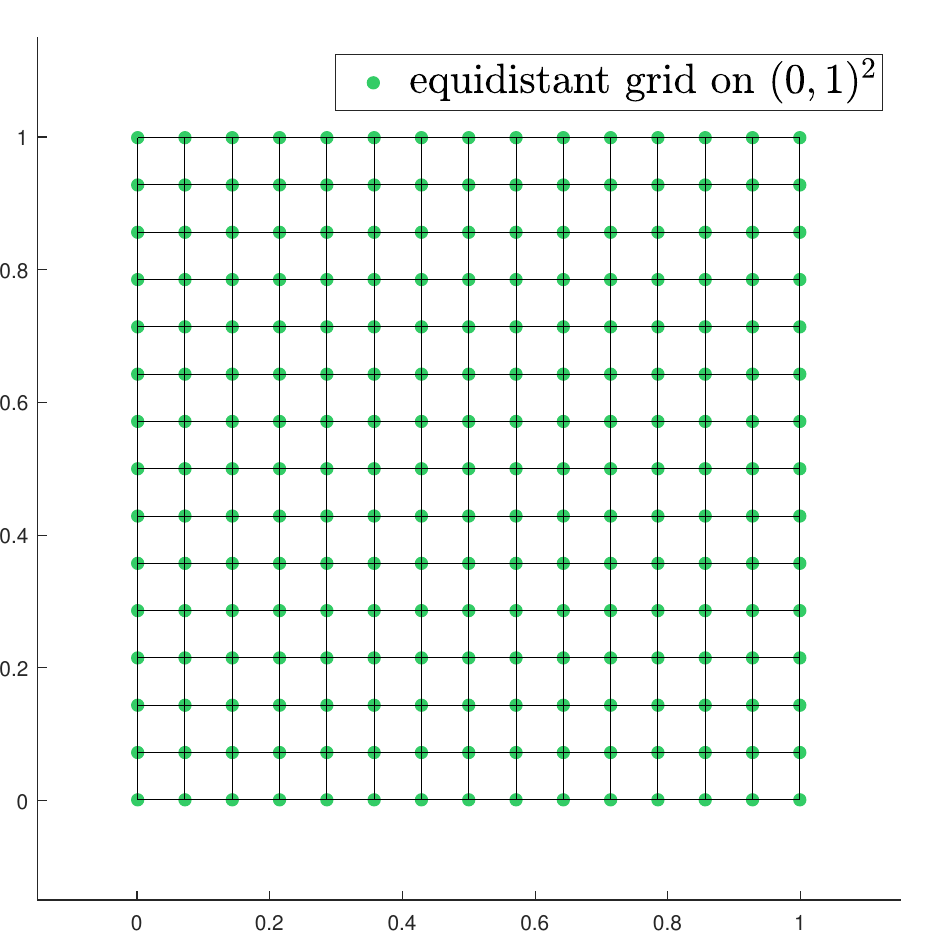}
	\end{subfigure}
	\hfill
	\begin{subfigure}[b]{0.28\textwidth}
		\Large
		\centering            
		$\xrightarrow[\text{$\Unif((0,1)^d)$ to $\cN(0,\Id_d)$ }]{\text{Transport map $T^\cN$ from}}$
		\vspace{2cm}
	\end{subfigure}
	\hfill
	\begin{subfigure}[b]{0.35\textwidth}
		\centering
		\includegraphics[width=\textwidth]{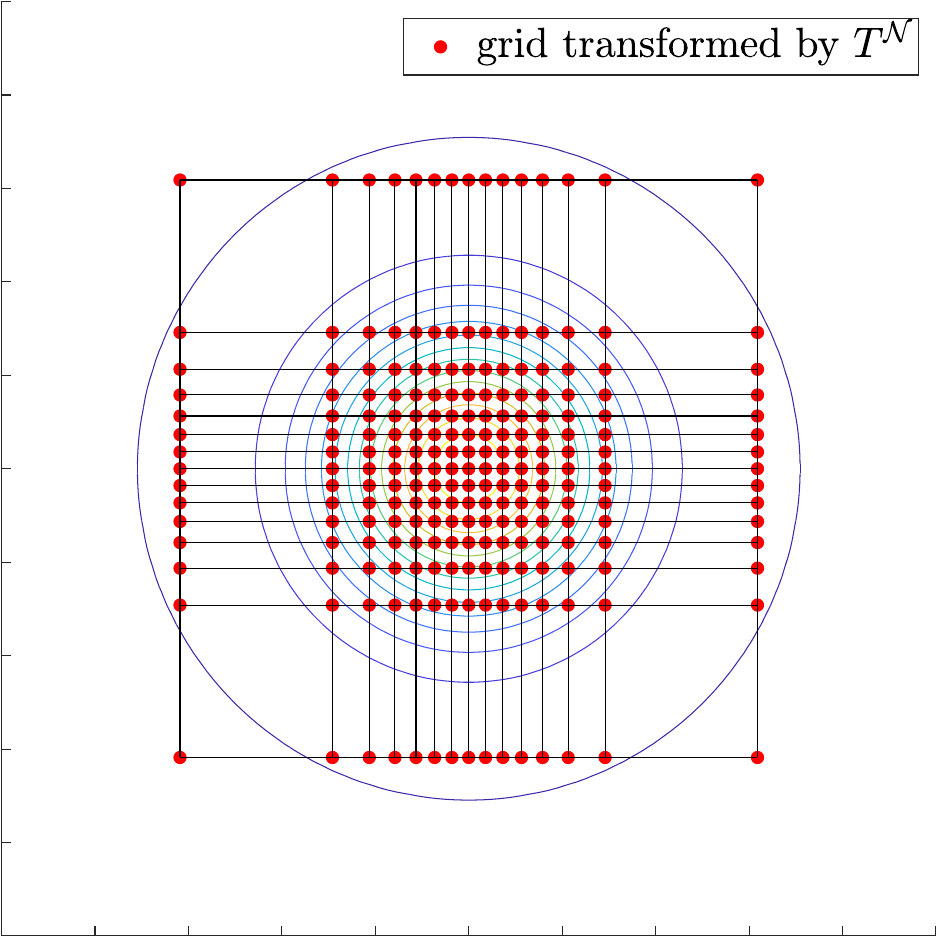}
	\end{subfigure}     
	\caption{In order to illustrate the transformation $T^\cN$ from \eqref{equ:TransportMapUniformToNormal} in \Cref{example:TransportMapUniformToNormal}, we apply it to an equidistant grid on $(0,1)^{2}$.
	(Since $T^\cN$ is defined on the open cube $(0,1)^d$, there must be no grid points on the boundary. Therefore, the grid is chosen as $G\times G$, where $G \defeq \varepsilon + \{ (1-2\varepsilon) k/(K-1) \mid k= 0,\dots,K-1 \}$, with $K=15$ and $\varepsilon = 10^{-3}$.)}
	\label{figure:TransportMapUniformToNormal}
\end{figure}

For more complicated distributions, e.g.\ Bayesian posteriors in high dimension, transport maps often have to be approximated, which we discuss in \Cref{section:ConstructionTransportMaps}.
In \Cref{section:TransportToMixtures} we will construct transport maps to certain mixture distributions.

As indicated in \Cref{remark:ConsequenceOfTransportMapsForSampling}, a transport map $T$ from $\bP_0$ to $\bP_1$ can be used to generate $\bP_1$-distributed samples once we are able to sample from $\bP_0$.
When $X_n \iidsim \bP_0$ for $n \in \bN$, it follows that $Y_n \defeq T(X_n) \iidsim \bP_1$.
Similarly to independent samples, one can consider transporting other point sequences such as sparse grids and \ac{QMC} points, in order to obtain a better convergence rate of certain numerical methods, e.g.\ quadrature.
We first give a general definition of convergence rates of (sequences of) deterministic or stochastic quadrature rules:

\begin{definition}
	\label{def:ConvergenceRate}
	Let $d\in\bN$, $\Omega\subseteq\bR^d$ be a domain in $\bR^d$, let $(\Omega,\cA,\bP)$ be a probability space, let $\cC\subseteq \bR^\Omega$ be a class of functions on $\Omega$, let $\mathfrak{N} = (N_{k})_{k\in\bN}$ be an increasing sequence in $\bN$ and let $\fw = (w^{(N_{k})})_{k\in\bN}$ be a sequence of weight vectors $w^{(N_{k})} = (w_{n}^{(N_{k})})_{n=1,\dots,N_{k}}$ (i.e.\ $w_{n}^{(N_{k})}\in \bR$ with $\sum_{n=1}^{N_{k}} w_{n}^{(N_{k})} = 1$ for each $k\in\bN$).
	We say that a sequence $\fX = (X^{(N_{k})})_{k \in \bN}$ of families $X^{(N_{k})} = (X_{n}^{(N_{k})})_{n = 1,\dots,N_{k}}$ in $\bR^{d}$, where $X_{n}^{(N_{k})}$ can either be deterministic points or random variables, has \emph{$(\bP,\cC,\mathfrak{N},\fw)$-convergence rate} $r\colon \bN\to\bR_{>0}$, if, for any $f\in \cC$,
	\begin{equation}
		\label{equ:DefinitionConvergenceRate}
		\abs{
			\bE_{X\sim \bP}[f(X)]
			-
			\sum_{n=1}^{N_{k}} w_{n}^{(N_{k})} f(X_{n}^{(N_{k})})
		}
		=
		\cO_{p}(r(N_{k})).
	\end{equation}
	In this case, we will refer to $\fX$ as \emph{\betterpoints with distribution $\bP$ and $(\bP,\cC,\mathfrak{N},\fw)$-convergence rate $r$}.
\end{definition}

\begin{remark}
	To simplify notation, we will often drop the $\mathfrak{N}$ dependence of the convergence rate and write $N$ in place of $N_{k}$ as though $\eqref{equ:DefinitionConvergenceRate}$ would hold for any $N\in\bN$.
	This is, however, an abuse of notation since the most advantageous performance of a quadrature rule is often attained only along specific sequences $\mathfrak{N}$, \ac{QMC} quadrature rules being a good example of this phenomenon \citep{owen2020sobol} as are \acp{SG}, the sizes of which result from the specific constructions after choosing a certain `level' \citep{novak1996high}.
	Furthermore, we will often write $X_{n}$ and $w_{n}$ instead of $X_{n}^{(N_{k})}$ and $w_{n}^{(N_{k})}$, respectively (in the case of \ac{QMC} points, $X_{n}^{(N_{k})}$ is, in fact, independent of $N_{k}$).
	Finally, we will often omit the prefix $(\bP,\cC,\mathfrak{N},\fw)$- if it is clear which quantities are being considered.	
\end{remark}

By the central limit theorem, MC samples are therefore \betterpoints with rate $r(N) = N^{-1/2}$.
Remarkably, this holds for the class $\cC = L^{2}(\bP)$ of all $\bP$-square-integrable functions and independently of the dimension $d$.
The convergence rates for sparse grids and QMC points are also well established and there are various results for different types of constructions. Here are two classical examples:

\begin{example}[Convergence rate for sparse grids]
	Consider $\Omega = [-1,1]^d$, $\bP = \Unif(\Omega)$ and the class of functions of regularity $\alpha\in\bN$,
	\begin{equation}
		\label{equ:integrand_class_sparse_grids}
		\cC
		=
		\cF_d^{\alpha}(\Omega)
		\defeq
		\{ f\colon \Omega \to\bR \mid D^\beta f \text{ is continuous if $\beta_i\le\alpha$ for all } i \}.
	\end{equation}
	Then the sparse grid $A(q,d)$ of \citet[Section~2]{novak1996high} (see also \citet{wasilkowski1995tensor}) has the following convergence rate \citep[Corollary~1]{novak1996high}:
	\begin{equation}
	\label{equ:Convergence_rate_SG}
		r(N)
		=
		N^{-\alpha} (\log N)^{(d-1)(\alpha + 1)}.
	\end{equation}
	Here, the sequences $\fw$ and $\fN$ of weight vectors and point numbers are defined by the corresponding sparse grid construction.
\end{example}

Similar convergence rates hold for sparse grids designed for other product measures $\bP$ on possibly unbounded domains $\Omega$, e.g.\ the sparse grids based on Gauss--Hermite points when $\bP$ is a standard normal distribution \citep{novak1997curse,novak1999interpolatory,novak1999cubature}.

\begin{example}[Convergence rate for \ac{QMC}]
	\label{example:Convergence_rate_QMC}
	Considering $\Omega = [0,1]^d$, $\bP = \Unif(\Omega)$ and the class of functions with a bounded variation in the sense of Hardy and Krause \citep{owen2005multidimensional}, $\cC = \BVHK(\Omega)$, the $d$-dimensional Halton sequence has the following convergence rate \citep[Section~5.6]{caflisch1998monte}:
	\begin{equation}
		\label{equ:Convergence_rate_QMC}
		r(N)
		=
		N^{-1} (\log N)^{d}.
	\end{equation}
	Here, the weights are uniform and $\fN = \bN$.
\end{example}

\begin{remark}
	\label{remark:recent_QMC_constructions}
	More recent QMC constructions, such as (randomly shifted) lattice rules, rely on weighted Sobolev and Korobov spaces acting as reproducing kernel Hilbert spaces rather than on function classes of bounded variation;
	see e.g.\ \citet{dick2013high} for an overview.
	Similar to \Cref{example:Convergence_rate_QMC}, convergence rates of almost $N^{-1}$ can be obtained in these cases;
	see \citet[Theorems~5.8, 5.9, 5.10, and 5.12]{dick2013high} for corresponding results.
\end{remark}

The following simple proposition justifies our approach of transporting \betterpoints:

\begin{proposition}[Transfer of convergence rate]
	\label{prop:TransportedConvergenceRate}
	Let $(\Omega_0,\cA_0,\bP_0)$ and $(\Omega_1,\cA_1,\bP_1)$ be two probability spaces, $\cC\subseteq \bR^{\Omega_0}$ and $T$ be a transport map from $\bP_0$ to $\bP_1$.
	If $(X_n)_{n\in\bN}$ has $(\bP_0,\cC)$-convergence rate $r\colon \bN\to\bR_{>0}$, then $(Y_n)_{n\in\bN}$ with $Y_n = T(X_n)$ also has $(\bP_1,\cC_T)$-convergence rate $r$, where $\cC_T \defeq \{ g\colon \Omega_1\to\bR \mid g\circ T\in\cC \}$.
\end{proposition}

\begin{proof}
	The statement follows directly from the change of variables formula \eqref{equ:ChangeOfVariablesFormula}.
\end{proof}

\begin{remark}
	\label{remark:transformed_class_of_functions}
	Note that we make no statement about the properties of the class of functions $\cC_T$.
	In the case of sparse grids with $\cC = \cF_d^{\alpha}(\Omega_0)$ as in \eqref{equ:integrand_class_sparse_grids}, if $T$ is a $C^{m}$-diffeomorphism with $m\ge d\alpha$, then $\cC_T \supseteq \cF_d^{\alpha}(\Omega_1)$ \citep[Section~7.4, Corollary~5]{arnold2006ode}.
	In the case of QMC and $\cC = \BVHK(\Omega_0)$ it is much harder to find suitable conditions on $T$ for $\cC_T \supset \BVHK(\Omega_1)$ to hold.
	Partial answers are given by \citet{basu2016transformations} and \citet{josephy1981composingBV}.
	For certain product measures on $\bR^{d}$, \citet{kuo2010lattice} and \citet{nichols2014lattice} discuss how the resulting function class $\cC_T \defeq \{ g\colon \bR^{d}\to\bR \mid g\circ T\in\cC \}$ can be characterised for the corresponding transport map $T\colon (0,1)^{d} \to \bR^{d}$, which in this case acts in a coordinatewise manner similar to $T^{\cN}$ in \Cref{example:TransportMapUniformToNormal}.
	The sophisticated derivations in these papers suggest that the analysis of the function class $\cC_{T}$ for the more complicated (in particular, non-coordinatewise) transport maps established in this manuscript justifies a separate paper.
	It remains an open problem for future research.
\end{remark}

\begin{remark}
	In most of our examples, $\trv$ will be a Gaussian mixture defined on the whole Euclidean space $\bR^{d}$, which is not homeomorphic to a closed cube.
	While for \ac{SG} there exist constructions based on Gauss--Hermite points which guarantee convergence rates similar to \eqref{equ:Convergence_rate_SG} on $\bR^{d}$, one has to be particularly careful when working with QMC constructions based on $\BVHK([0,1]^{d})$.
	However, more recent QMC approaches mentioned in \Cref{remark:recent_QMC_constructions}, such as shifted lattice rules based on certain Sobolev spaces rather than on $\BVHK$, are more flexible in this regard, allowing to work with the open cube $(0,1)^{d}$ in place of the closed one.
\end{remark}
	\section{Construction of Transport Maps: Background}
\label{section:ConstructionTransportMaps}

Explicit transport maps from $\simpv$ to $\trv$ can only be constructed in very special cases such as the one discussed in \Cref{example:TransportMapUniformToNormal}.
In most other cases, numerical approximations are necessary.
There are several approaches to the construction of approximate transport maps.
One is to define a finite-dimensional class of possible maps and to optimise a certain functional with respect to the parameters \citep{el2012bayesian,marzouk2016introduction}.
Another approach, which we are going to focus on in this manuscript, is the construction of transport maps via \ac{ODE}s.
The idea is to define an \ac{ODE}
\begin{equation}
	\label{equ:generalODE}
	\dot{x}(t) = v_t(x(t)),
	\qquad
	t\in [0,1],
\end{equation}
with flow maps $\Phi^t$, $t\in[0,1]$, such that $\Phi^{1}$ is the desired transport map from $\simpv$ to $\trv$.
If $x(0) = x_0\sim \simpv$, then consequently $x(1) = \Phi^{1}(x_0)\sim \trv$ by \Cref{remark:ConsequenceOfTransportMapsForSampling}.

For this purpose, a typical strategy is to define a family of `intermediate' probability densities $(\rho_t)_{t\in[0,1]}$ with $\rho_0 = \simpd$ and $\rho_1 = \trd$, e.g.\ via $\rho_t\propto\simpd^{1-t}\,\trd^t$, and find a family of velocity fields $(v_t)_{t\in[0,1]}$ that satisfies the continuity equation
\begin{equation}
	\label{equ:PlainContinuityEquation}
	\partial_t \rho_t
	=
	-\diver(\rho_t v_t),
\end{equation}
which guarantees that $x(t)$ has probability density $\rho_{t}$ for each $t\in [0,1]$, in particular $\Phi_{\#}^{1} \simpv = \trv$.
See \citet[Proposition~2.1]{ambrosio2008transport} and \citet[Proposition~8.1.8]{ambrosio2008gradient} for a precise statement and proof of the continuity equation and its consequences.

The main difficulty of this strategy, which has been implemented by e.g.\ \citet{heng2015gibbs}, lies in approximating the velocity fields $v_t$ that solve the continuity equation \eqref{equ:PlainContinuityEquation}, especially in high dimension.
In \Cref{section:TransportToMixtures} we will construct an explicit formula for $v_t$ for certain mixture distributions $\trv$ and (intermediate) probability densities $(\rho_t)_{t\in[0,1]}$.

\begin{remark}
	There is another class of methods that construct an \ac{ODE} with the property $\Phi_{\#}^\tau \simpv = \trv$ for either $\tau = 1$ or $\tau = \infty$, without (a priori) defining intermediate probability densities $(\rho_t)_{t\in[0,1]}$.
	Instead, the velocity field $v_t$ is chosen in such a way that $\rho_t$ gets closer to $\trd$ in some metric, i.e.\ $x(t)$ is `pushed' in the `right' direction.
	This approach comes with its own challenges such as obtaining some knowledge about the current probability density $\rho_{t}$, with \acl{SVGD} (\acs{SVGD}; \citealp{liu2016stein}) providing one elegant solution.
	Typically, the constructed ODE is solved for a large number of particles simultaneously and requires some form of particle interaction, which can break the delicate structure of many higher-order point constructions.
	The present paper does not follow this approach.
\end{remark}
	
\section{Transport Maps to Mixture Distributions}
\label{section:TransportToMixtures}

Let $\simpv$ be a simple probability distribution with positive density function $\simpd\colon\bR^d\to\bR_{>0}$.
By `simple' we mean that we can produce $\simpv$-distributed \betterpoints with some given rate in the sense of \Cref{def:ConvergenceRate}, e.g.\ sparse grids for $\simpv = \cN(0,\Id_d)$, or that there exists an explicit transport map from such a distribution to $\simpv$.
In this section we discuss how to build a transport map from $\simpv$ to a mixture $\trv$ of shifted and scaled versions of $\simpv$ given by the probability density
\begin{equation}
	\label{equ:MistureDensityScaled}
	\trd (x) = \sum_{j=1}^{J} w_j\, \abs{\det A_j}^{-1}\, \simpd(A_j^{-1}(x-a_j)),
	\quad
	a_j\in\bR^d,
	\quad
	A_j\in \mathbf{S}_{++}^{d},
	\quad
	\sum_{j=1}^{J} w_j = 1,
\end{equation}
in order to obtain \betterpoints from $\trv$ in line with \Cref{prop:TransportedConvergenceRate}.
Here, $\mathbf{S}_{++}^{d}$ denotes the set of symmetric and positive definite matrices $A \in \bR^{d\times d}$.
We will mostly assume the weights $w_j\in\bR$ to be non-negative, while arbitrary weights will be discussed in \Cref{section:MixturesWithNegativeWeights} in the special case of Gaussian mixtures.
Making use of \Cref{lemma:LinearChangeOfVariablesForDensities}, it is straightforward to produce independent samples $Y$ from $\trv$ by the so-called \emph{composition method} \citep[Section~2.3.3]{rubinstein2016simulation}:

\begin{algorithm}[H]
	\caption{Composition method for direct sampling from mixture distributions.}
	\label{alg:direct_sampling_of_mixtures}
	\begin{algorithmic}[1]
		\State
		\textsc{Input}: $\trv$ given by \eqref{equ:MistureDensityScaled}.
		\State
		Draw $X\sim \simpv$.
		\State 
		Draw $\nu$ from the discrete distribution\footnote{Sampling from a discrete distribution $w$ is straightforward:
			Draw $Z\sim\Unif([0,1])$ and set $\nu = j$ whenever $\sum_{k=1}^{j-1} w_k < Z \le \sum_{k=1}^{j} w_k$.} $w = (w_1,\dots,w_J)$, i.e.\ $\bP(\nu = j) = w_j$.
		\State
		\textsc{Output}: $Y = A_{\nu} X + a_{\nu}$.
	\end{algorithmic}
\end{algorithm}

However, as discussed in \Cref{section:Intro}, independent samples $Y_1,\dots,Y_N \iidsim \trv$ provide a rather slow convergence rate of $N^{-1/2}$ in the \ac{MC} approximation \eqref{MCapproximation}, which can be improved by transporting \betterpoints with a better rate from $\simpv$ to $\trv$.

The main result of this paper is the construction of an explicit transport map from $\simpv$ to $\trv$ if $\trv$ is given by \eqref{equ:MistureDensityScaled}.
More precisely, the construction relies on the solution of an explicit \ac{ODE}, as discussed in \Cref{section:ConstructionTransportMaps}, with \emph{analytic} (i.e.\ closed-form) right-hand side.
The continuity equation \eqref{equ:PlainContinuityEquation} guarantees 
the correctness of the method.

\begin{theorem}[Transport to mixtures]
	\label{thm:TransportMixtures}
	Let $\simpd \in C^1(\bR^d;\bR_{>0})$ be a strictly positive probability density function with finite first moment, $M \defeq \int \norm{x} \, \simpd(x)\, \mathrm dx  < \infty$, $\trd$ be given by \eqref{equ:MistureDensityScaled} with $w_j\ge 0$, $j=1,\dots,J$, and $\simpv$ and $\trv$ denote the corresponding probability distributions.
	Further, for $j=1,\dots,J$ and $t\in[0,1]$, let
	\begin{align*}
		A_{j,t}
		& \defeq
		t A_j + (1-t) \Id_d,
		\\
		\rho_{j,t}(x)
		& \defeq
		\abs{\det A_{j,t}}^{-1} \simpd(A_{j,t}^{-1}(x-ta_j)),
		\\
		v_{j,t}(x)
		& \defeq
		a_j + (A_j - \Id_d)A_{j,t}^{-1}(x-ta_j),
	\end{align*}
	and let the densities $\rho_{t} \defeq \sum_{j=1}^{J}w_j\rho_{j,t}$, $t\in [0,1]$ define `intermediate' probability distributions $\bP_{t}$ on $\bR^{d}$.
	Then, for $\simpv$-almost every $x_{0} \in \bR^{d}$, the ODE
	\begin{equation}
		\label{equ:Mixture_transport_ODE}
		\dot{x}(t) = v_t(x(t)),
		\qquad
		x(0) = x_{0},
		\qquad
		v_t
		\defeq
		\rho_{t}^{-1}\sum_{j=1}^{J}w_j\rho_{j,t} v_{j,t},
		\qquad
		t\in [0,1],
	\end{equation}
	admits a globally defined solution on $[0,1]$ and, for every $t\in [0,1]$, the corresponding flow map $\Phi_t\colon\bR^d\to\bR^d$ satisfies $(\Phi_t)_{\#} \simpv = \bP_{t}$.
	In particular, $T = \Phi_1$ is a transport map from $\simpv$ to $\trv$.
\end{theorem}

\begin{proof}
	The proof is given in \Cref{section:Appendix}.
\end{proof}

The application of the above transport approach to \betterpoints is summarised in the following algorithm, which is the primary method suggested and analysed in this paper.

\begin{algorithm}[H]
	\caption{ODE transport for \betterpoints from mixture distributions.}
	\label{alg:ODE_sampling_of_mixtures}
	\begin{algorithmic}[1]
		\State
		\textsc{Input}: set of  \betterpoints{}\footref{footnote:initial_points} 
		(e.g.\ QMC points) $(X_{n})_{n = 1,\dots,N}$ in $\bR^{d}$; $\trv$ given by \eqref{equ:MistureDensityScaled}.
		\State
		For each $n = 1,\dots,N$, solve the ODE \eqref{equ:Mixture_transport_ODE} with initial point $x(0) = X_{n}$ and set $Y_{n} \defeq x(1)$.
		\State
		\textsc{Output}: $(Y_{n})_{n = 1,\dots,N}$,
	\end{algorithmic}
\end{algorithm}

\footnotetext{\label{footnote:initial_points} We assume these points to be $\simpv$-distributed in the sense of \Cref{def:ConvergenceRate}.
As mentioned in the beginning of this section, this might require a prior application of some explicit transport map (e.g.\ $T^{\cN}$ from \Cref{example:TransportMapUniformToNormal} if $\simpv = \cN(0,\Id_{d})$).}

\begin{corollary}
	Let $d\in\bN$, let $\Omega = \bR^d$, and let the sequence $\fX = (X^{(N_{k})})_{k \in \bN}$ of families $X^{(N_{k})} = (X_{n}^{(N_{k})})_{n = 1,\dots,N_{k}}$ in $\bR^{d}$ have $(\bP,\cC,\fN,\fw)$-convergence rate $r\colon \bN\to\bR_{>0}$ for some quantities $\bP,\cC,\fN,\fw$ as in \Cref{def:ConvergenceRate}.
	Then $\fY = (Y^{(N_{k})})_{k\in\bN}$ produced by \Cref{alg:ODE_sampling_of_mixtures} (when applied to each point set $X^{(N_{k})}$, $k\in\bN$, separately) has $(\bP,\cC_{T},\fN,\fw)$-convergence rate $r$.
\end{corollary}

\begin{proof}
	This is a direct consequence of \Cref{prop:TransportedConvergenceRate,thm:TransportMixtures}.
\end{proof}

\begin{figure}[p]
	\centering
	\begin{subfigure}[b]{0.32\textwidth}
		\centering
		\includegraphics[width=\textwidth]{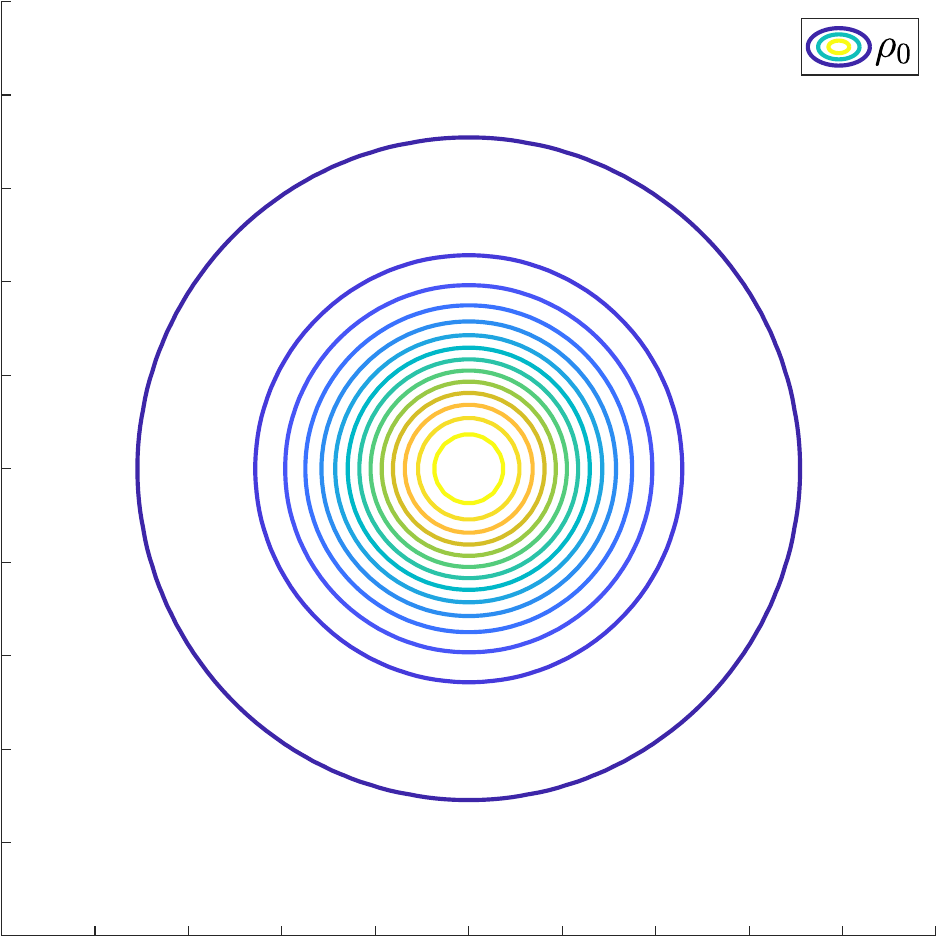}
	\end{subfigure}
	\hfill
	\begin{subfigure}[b]{0.32\textwidth}
		\large
		\centering            
		$\xrightarrow[\text{constructed via the ODE \eqref{equ:Mixture_transport_ODE}}]{\text{Transport map $T = \Phi_{1}$ from $\bP_0$ to $\bP_1$}}$
		\vspace{2cm}
	\end{subfigure}
	\hfill
	\begin{subfigure}[b]{0.32\textwidth}
		\centering
		\includegraphics[width=\textwidth]{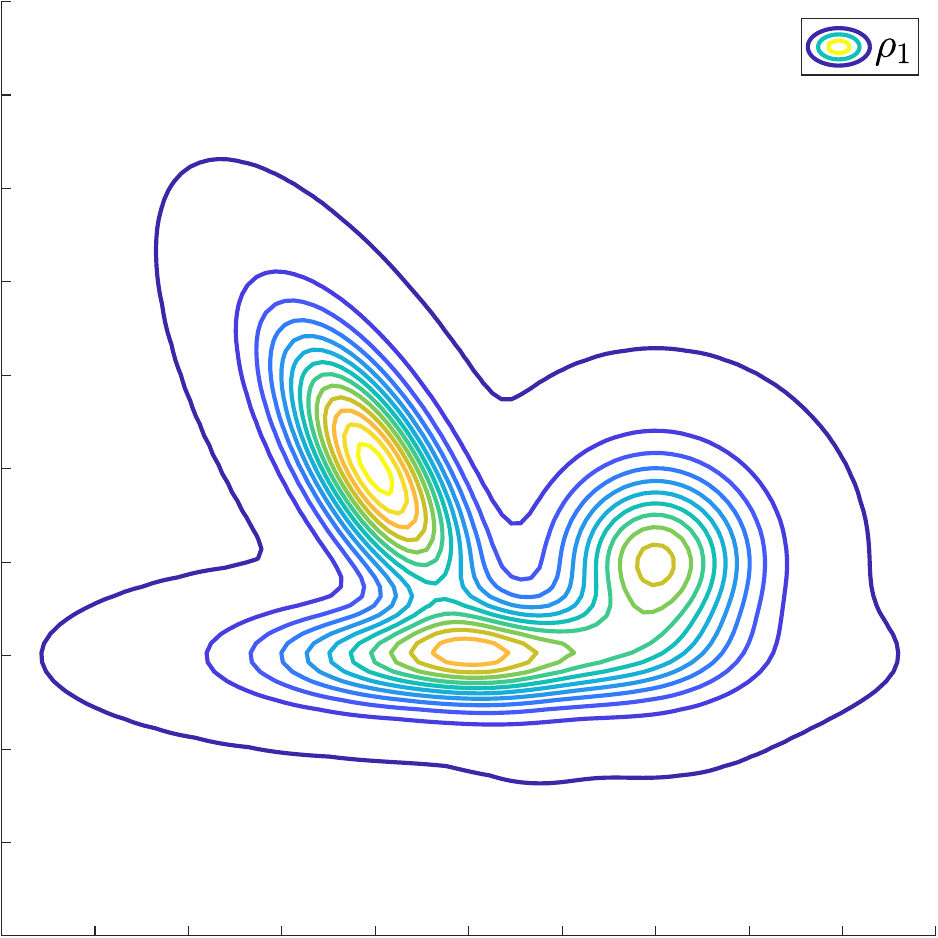}
	\end{subfigure}
	\vfill
	\begin{subfigure}[b]{0.32\textwidth}
		\centering
		\includegraphics[width=\textwidth]{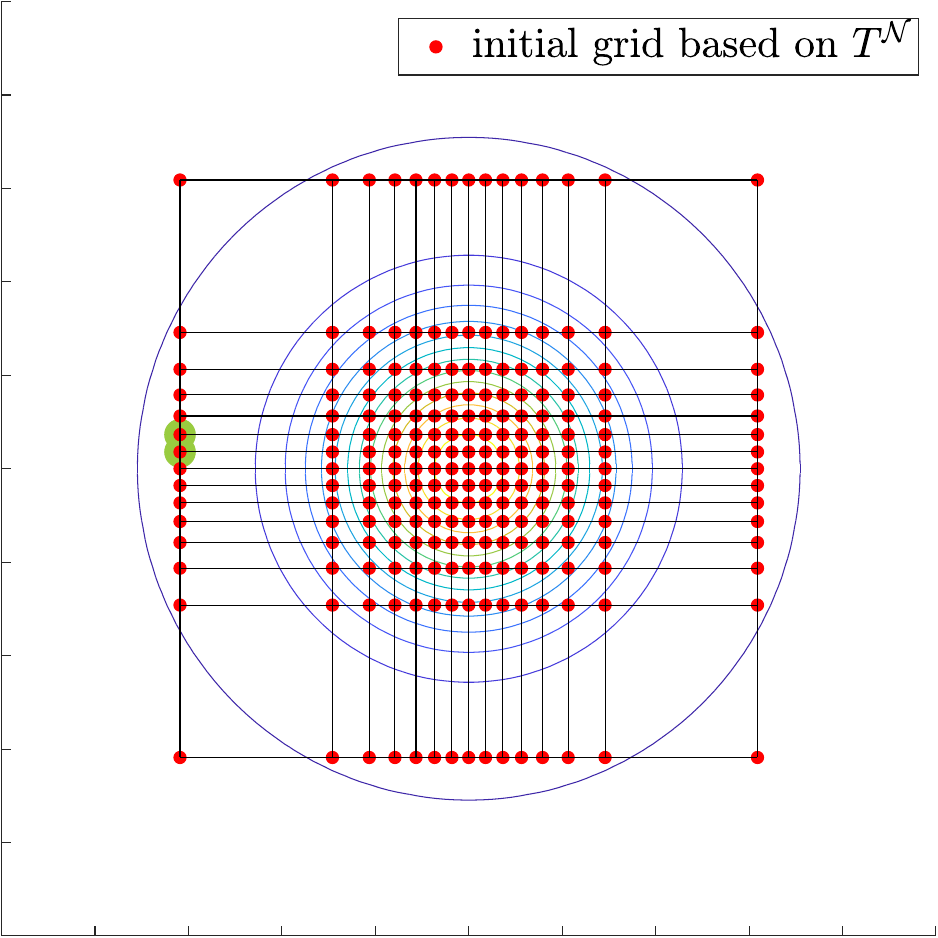}
	\end{subfigure}
	\hfill
	\begin{subfigure}[b]{0.32\textwidth}
		\centering
		\includegraphics[width=\textwidth]{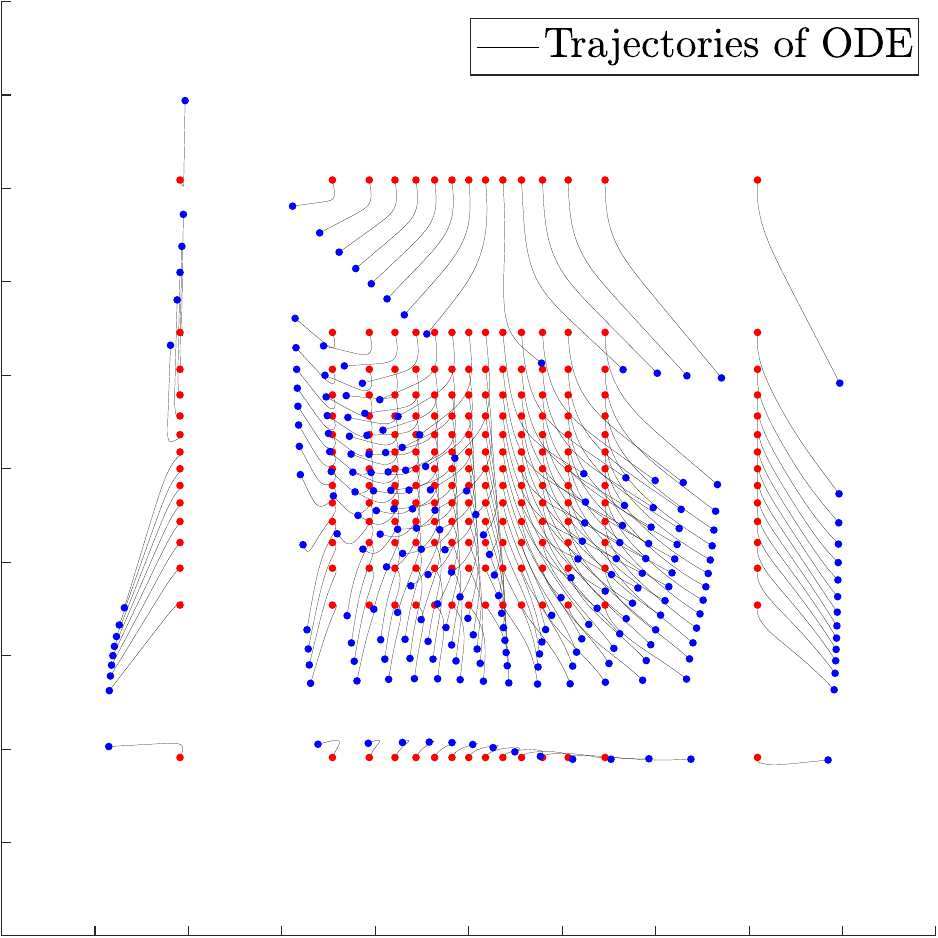}
	\end{subfigure}
	\hfill
	\begin{subfigure}[b]{0.32\textwidth}
		\centering
		\includegraphics[width=\textwidth]{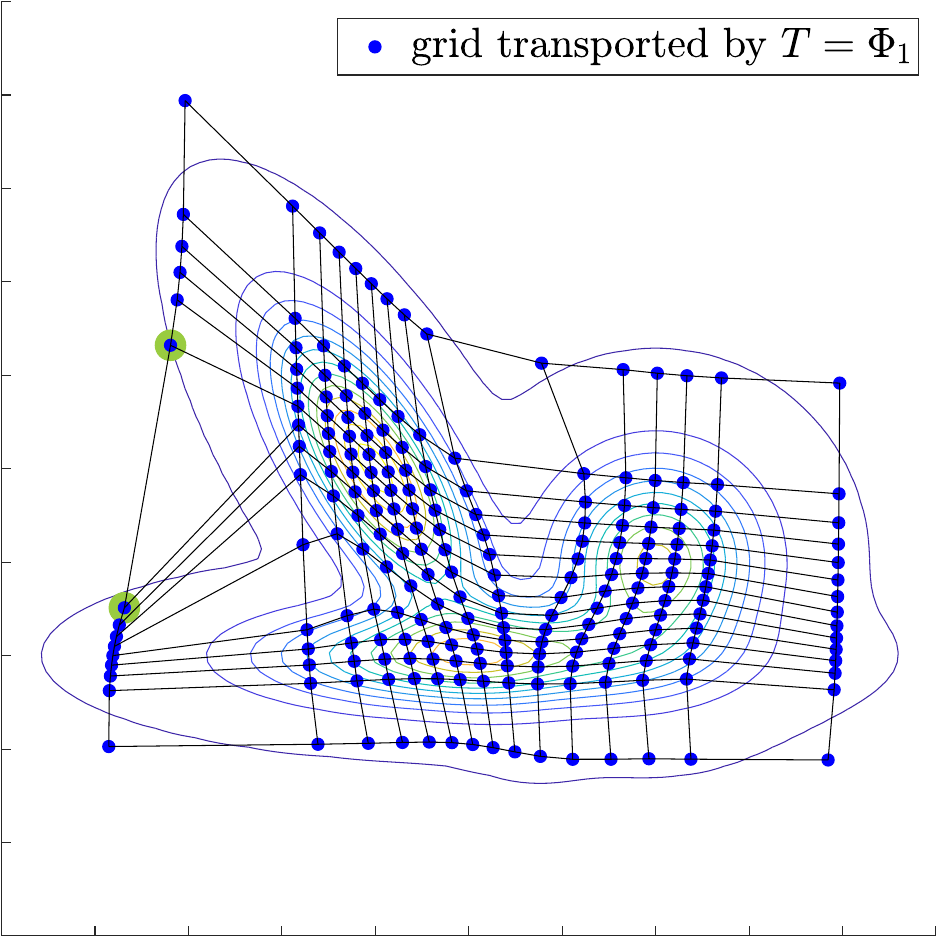}
	\end{subfigure}
	\vfill
	\begin{subfigure}[b]{0.32\textwidth}
		\centering
		\includegraphics[width=\textwidth]{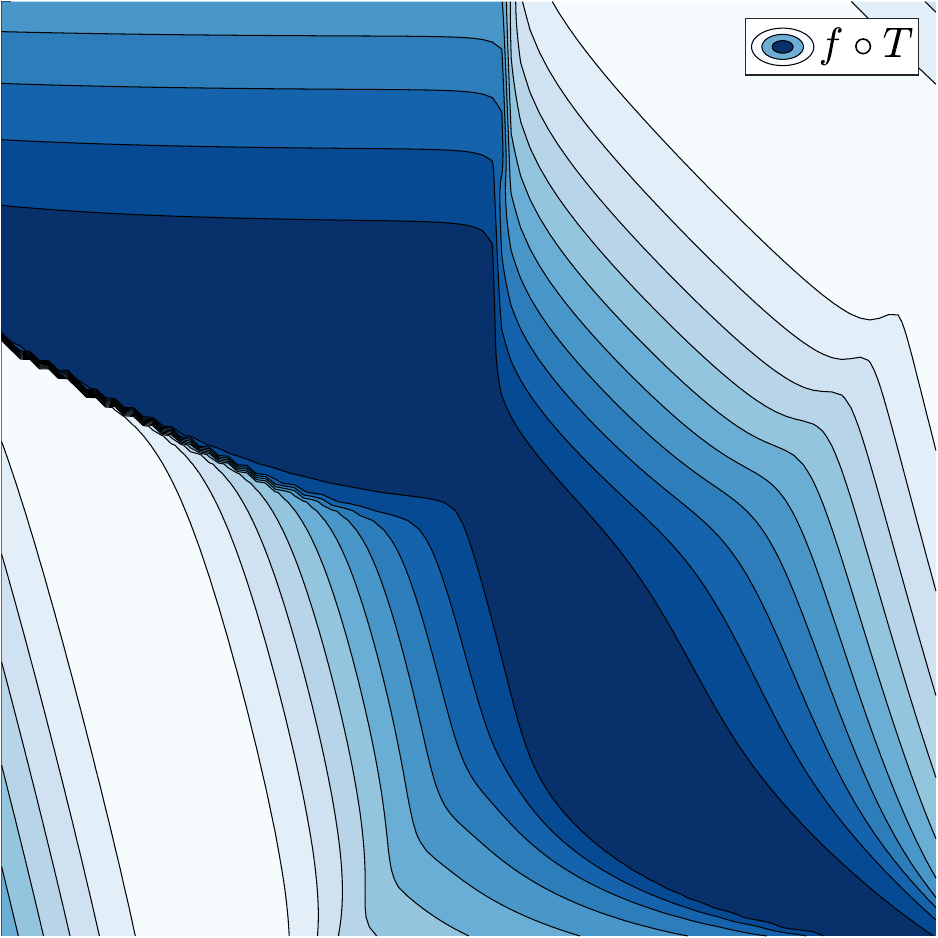}
	\end{subfigure}
	\hfill
	\begin{subfigure}[b]{0.32\textwidth}
		\large
		\centering            
		$\xleftarrow[]{\text{Function $f$ `pulled back' by $T = \Phi_{1}$}}$
		\vspace{2cm}
	\end{subfigure}
	\hfill
	\begin{subfigure}[b]{0.32\textwidth}
		\centering
		\includegraphics[width=\textwidth]{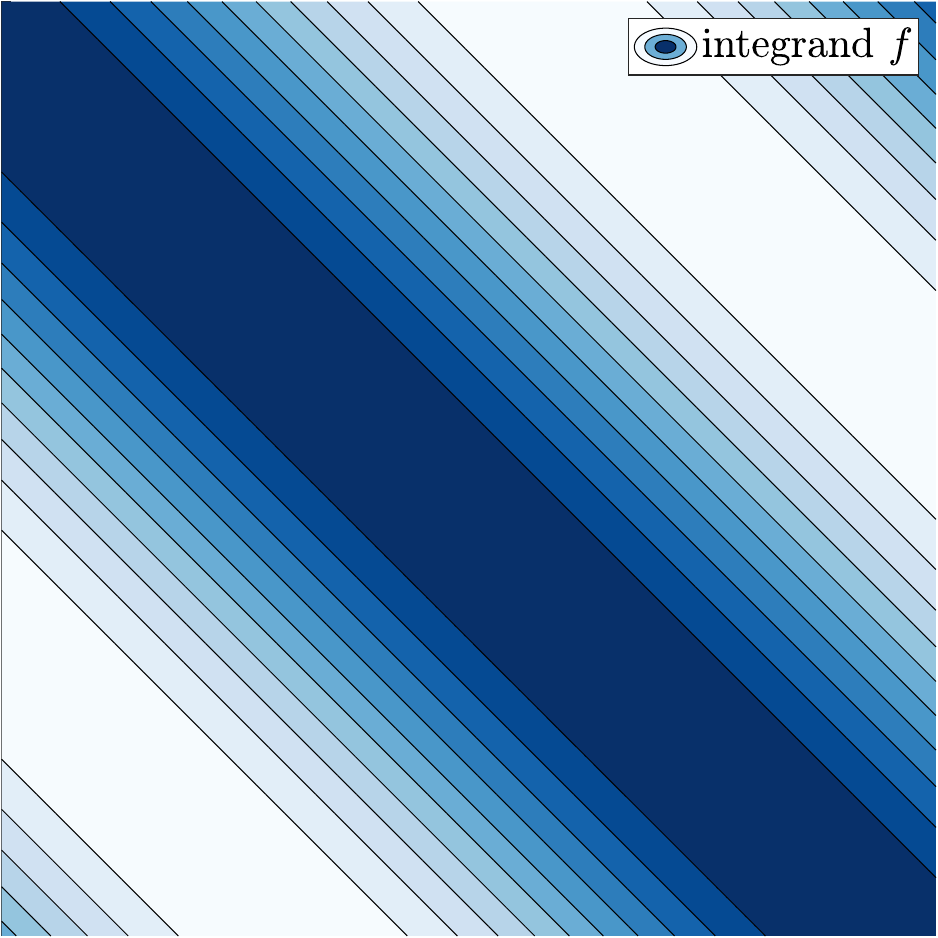}
	\end{subfigure}        
	\caption{We illustrate the transport map $T = \Phi_{1}$ established in \Cref{thm:TransportMixtures} by transporting the grid constructed in \Cref{example:TransportMapUniformToNormal} (see \Cref{figure:TransportMapUniformToNormal}) to a mixture of three scaled and shifted Gaussian densities given by \eqref{equ:Example_mixture_three_Gaussians}.
		By the dual viewpoint \eqref{equ:ChangeOfVariablesFormula}, approximating $\bE_{\trv}[f(Y)]$ using the transported points $Y_{n} = T(X_n)$ is equivalent to approximating $\bE_{\simpv}[f\circ T]$ using the original points $X_{n}$.
		We illustrate how the function $f(y) = \cos( 0.3 + (y_{1} + y_{2})/2)$ is pulled back to the function $f\circ T$ via the same transport map $T = \Phi_{1}$.
		Note how the two neighbouring grid points marked in green can end up rather far apart after being transported.
		As a result, the function $f\circ T$ becomes almost discontinuous in the corresponding region, cf.\ \Cref{rem:pulled_back_function_almost_discontinuous}.
	}
	\label{fig:transport_map_to_Gaussian_mixture_grid_and_pulled_back_function}
\end{figure}

Again, we wish to stress that, once $J$ is no longer considerably smaller than $N$, transporting $N$ \betterpoints using this scheme is preferable to transporting $J$ collections of (roughly) $\lfloor w_{j} N \rfloor$ \betterpoints, one collection per mixture component;
see also \Cref{remark:QMC_points_from_each_component_separately} and \Cref{fig:MC_QMC_error_over_J_dim_2}.
Such a `componentwise' approach was recently suggested by \citet{Cui2023quasimonte} and we sketch the corresponding algorithm here, which relies on \Cref{lemma:LinearChangeOfVariablesForDensities}:

\begin{algorithm}[H]
	\caption{Componentwise transport for \betterpoints from mixture distributions.}
	\label{alg:componentwise_transport_sampling_of_mixtures}
	\begin{algorithmic}[1]
		\State
		\textsc{Input}: set of \betterpoints{}\footref{footnote:initial_points}
		(e.g.\ QMC points) $(X_{n})_{n = 1,\dots,N}$ in $\bR^{d}$; $\trv$ given by \eqref{equ:MistureDensityScaled}.
		\State 
		Define the `Diophantine approximation' $M_{j} = \lfloor  w_{j}N \rfloor$, $j \leq J-1$, and $M_{J} = N-\sum_{j=1}^{J-1} M_j$.%
		\State
		\textsc{Output}: point set $\set{Y_{n}^{(j)} \defeq A_{j}X_{n} + a_{j}}{j=1,\dots,J,\ n=1,\dots,N_{j}}$, where each point $Y_{n}^{(j)}$ has the weight $w_{j}/M_{j}$.
	\end{algorithmic}
\end{algorithm}

\begin{terminology}
	We refer to the points (or sequences of point families) produced by \Cref{alg:ODE_sampling_of_mixtures} as \emph{transported \betterpoints}, in particular as \emph{transported quasi-Monte Carlo} (\ac{TQMC}) points and \emph{transported sparse grids} (\ac{TSG}).
	Correspondingly, \emph{componentwise \betterpoints}, \emph{componentwise quasi-Monte Carlo} (CQMC) points, and \emph{componentwise sparse grids} (CSG) will refer to the output of \Cref{alg:componentwise_transport_sampling_of_mixtures}.
\end{terminology}

\begin{figure}[p]
	\centering         	
	\begin{subfigure}[b]{0.32\textwidth}
		\centering
		\includegraphics[width=\textwidth]{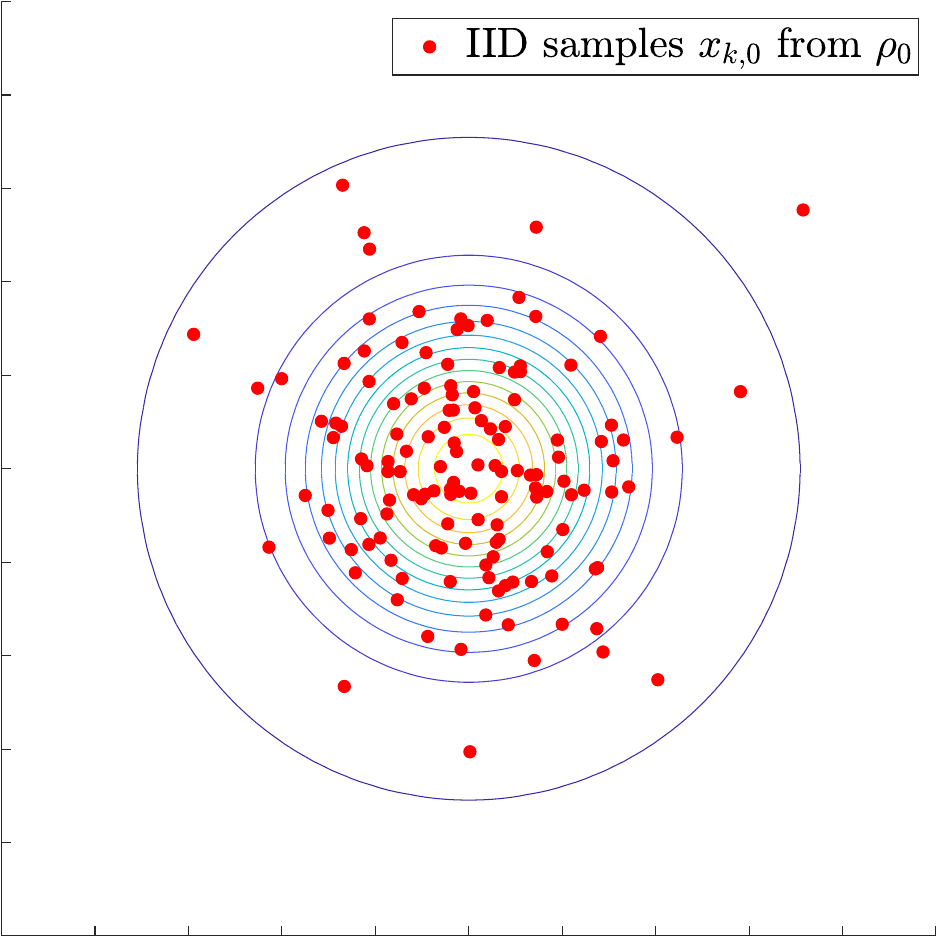}
	\end{subfigure}
	\hfill
	\begin{subfigure}[b]{0.32\textwidth}
		\centering
		\includegraphics[width=\textwidth]{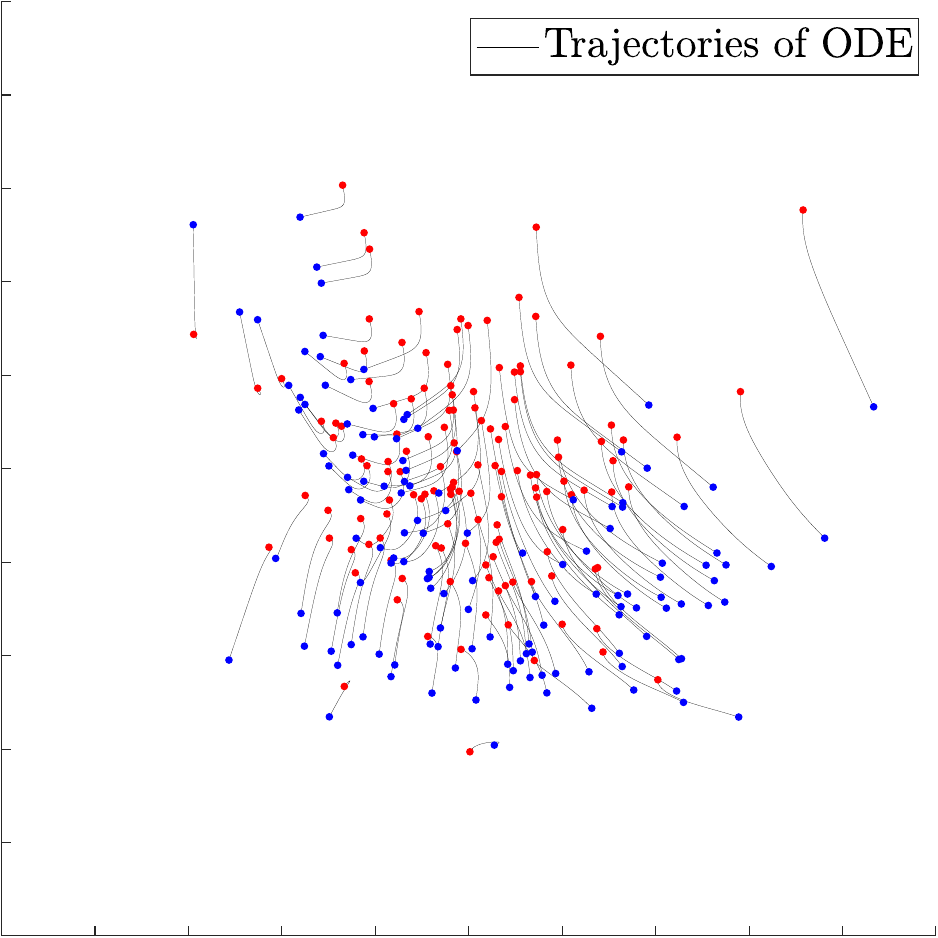}
	\end{subfigure}
	\hfill
	\begin{subfigure}[b]{0.32\textwidth}
		\centering
		\includegraphics[width=\textwidth]{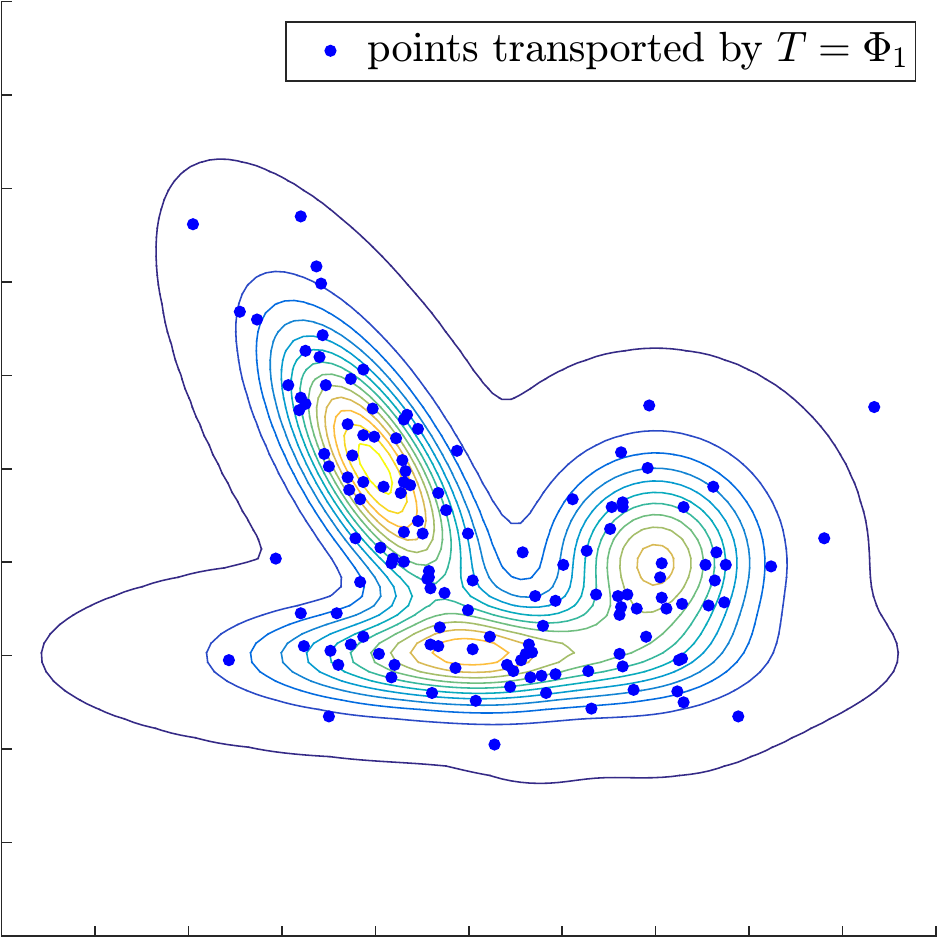}
	\end{subfigure}
	\vfill
	\begin{subfigure}[b]{0.32\textwidth}
		\centering
		\includegraphics[width=\textwidth]{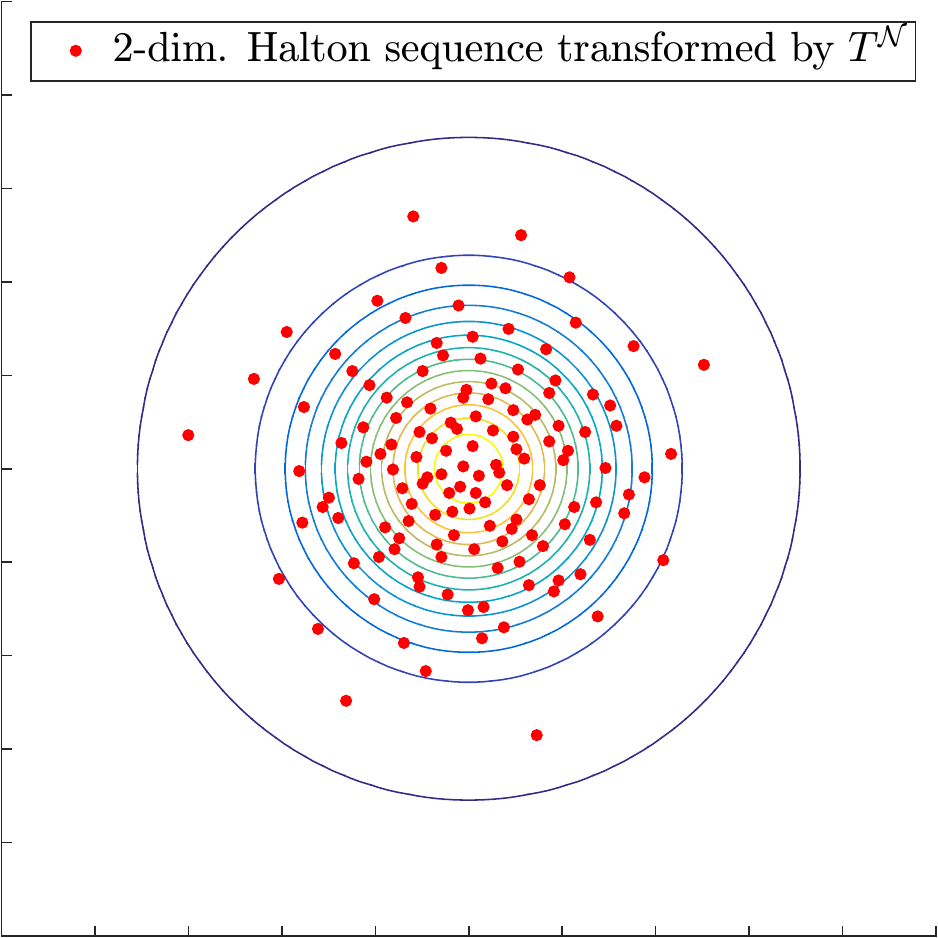}
	\end{subfigure}
	\hfill
	\begin{subfigure}[b]{0.32\textwidth}
		\centering
		\includegraphics[width=\textwidth]{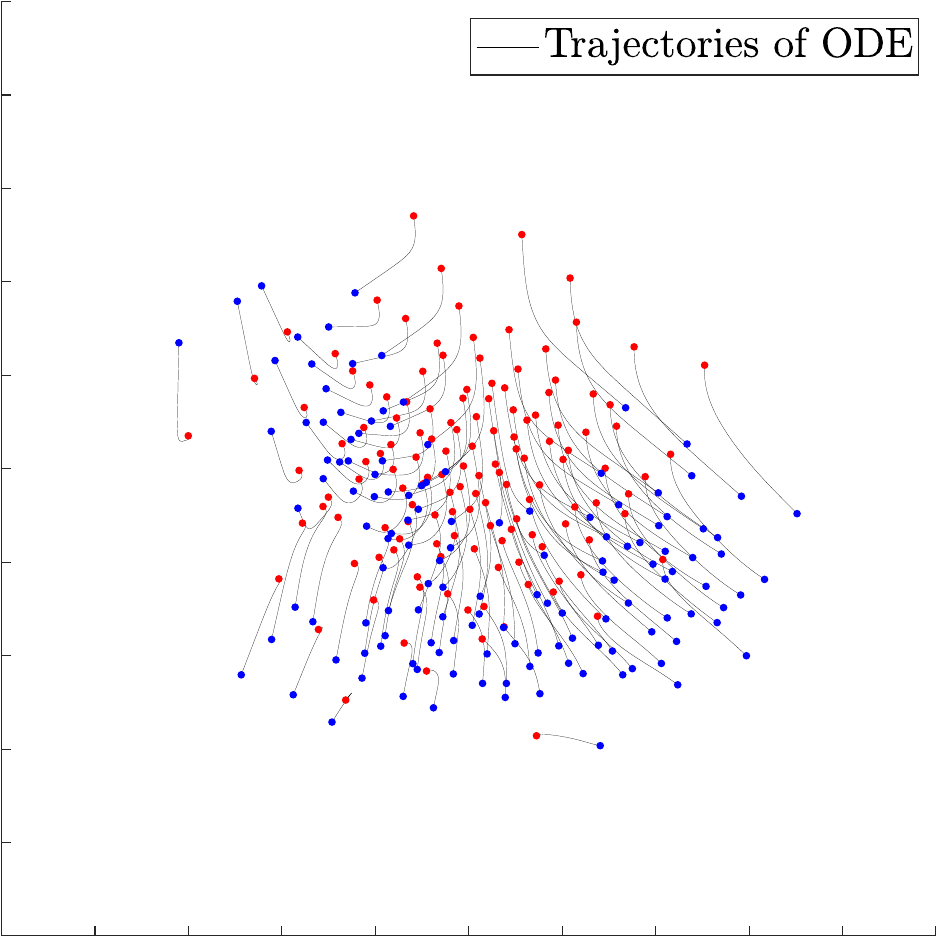}
	\end{subfigure}
	\hfill
	\begin{subfigure}[b]{0.32\textwidth}
		\centering
		\includegraphics[width=\textwidth]{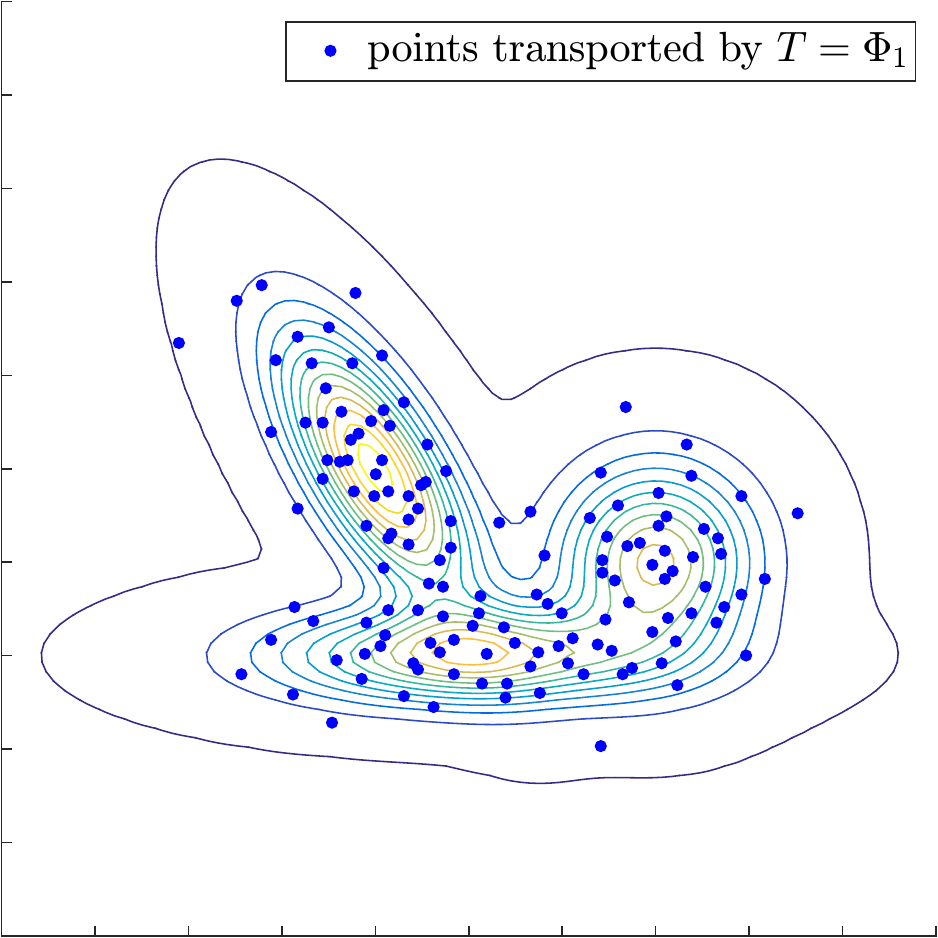}
	\end{subfigure}
	\vfill
	\begin{subfigure}[b]{0.32\textwidth}
		\centering
		\includegraphics[width=\textwidth]{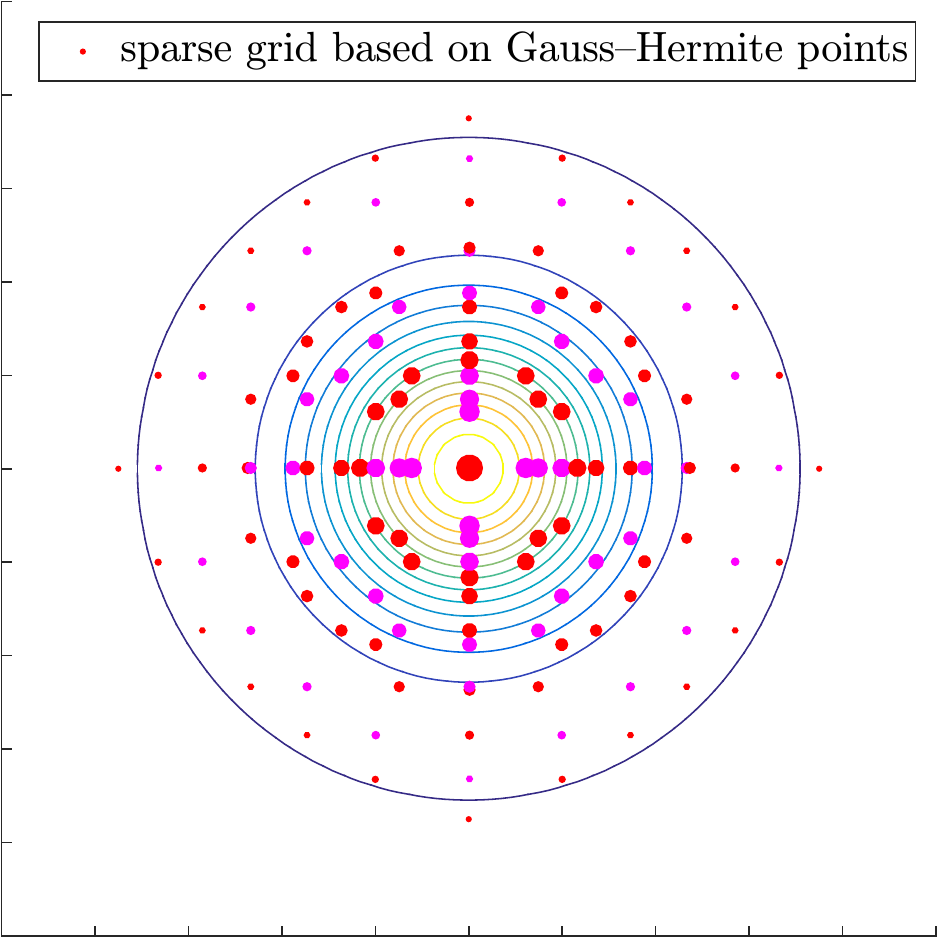}
	\end{subfigure}
	\hfill
	\begin{subfigure}[b]{0.32\textwidth}
		\centering
		\includegraphics[width=\textwidth]{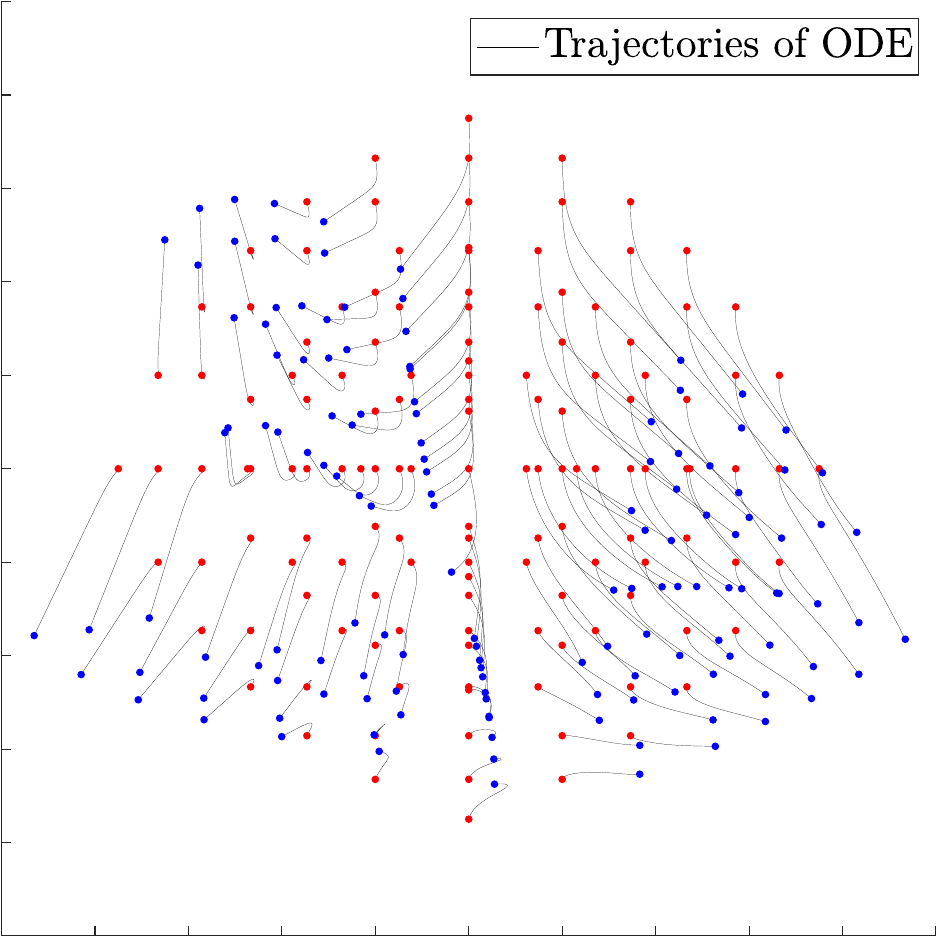}
	\end{subfigure}
	\hfill
	\begin{subfigure}[b]{0.32\textwidth}
		\centering
		\includegraphics[width=\textwidth]{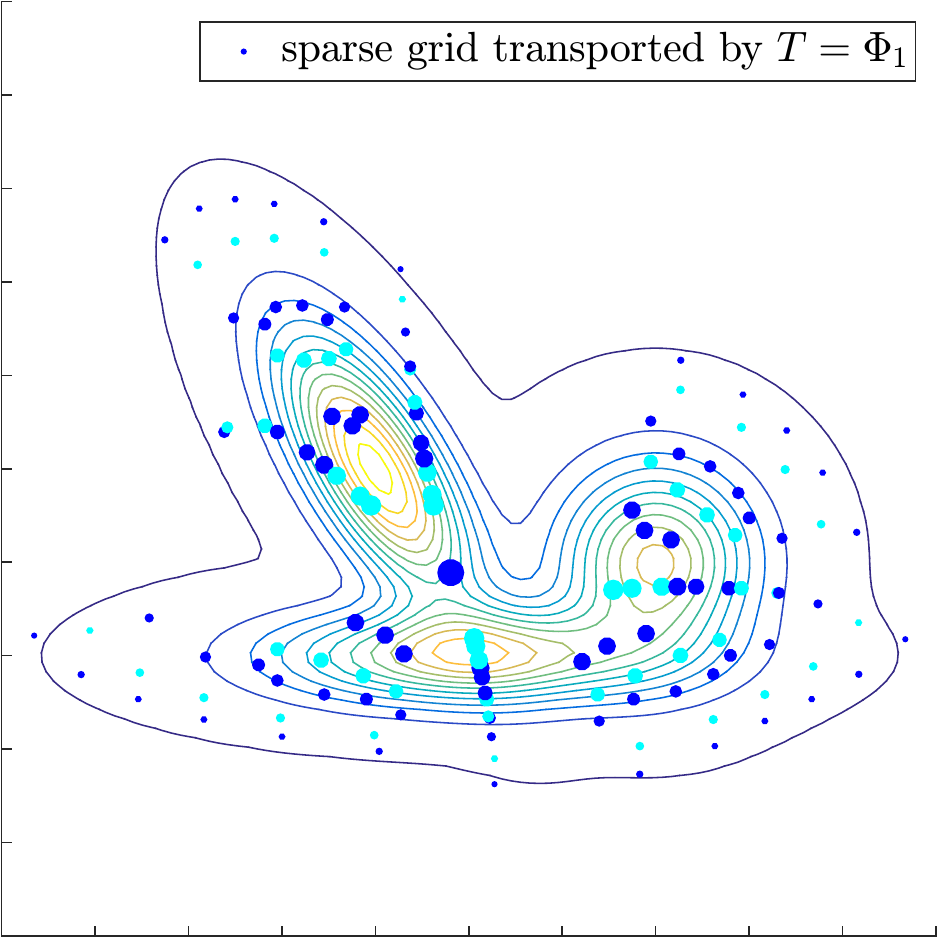}
	\end{subfigure}
	\caption{Transporting different types of \betterpoints from the initial distribution $\simpv$ to the Gaussian mixture target $\trv$ given by \eqref{equ:Example_mixture_three_Gaussians} by \Cref{alg:ODE_sampling_of_mixtures}: Monte-Carlo points (\ac{IID}), \ac{QMC} points (2-dimensional Halton sequence transformed by $T^{\cN}$) and sparse grids (level 6, based on Gauss--Hermite points to match the initial distribution).
	For sparse grids, the color and size of a point indicate its weight (red and blue for positive weights, magenta and cyan for negative weights).
	In all three cases, a total number of 137 points were transported.
	}
	\label{fig:TransportingMCqMCsparseGridsToMixture}
\end{figure}

\begin{example}
	\label{example:Transport_mixture_three_Gaussians}
	Consider a Gaussian mixture distribution $\trv$ given by \eqref{equ:MistureDensityScaled} with standard normal distribution $\simpv$ in dimension $d=2$, $J=3$, $w = (\nicefrac{3}{10} , \nicefrac{4}{10} , \nicefrac{3}{10})$ and
	\begin{equation}
		\label{equ:Example_mixture_three_Gaussians}
		\begin{split}
			\begin{aligned}
				a_1 & = \begin{pmatrix} 2 \\ -1\end{pmatrix},
				&
				a_2 & = \begin{pmatrix} -1 \\ 0\end{pmatrix},
				&
				a_3 & = \begin{pmatrix} 0 \\ -2\end{pmatrix},
				\\
				A_1
				& =
				\frac{2}{3}\, 
				\begin{pmatrix}
					1 & 0\\ 0 & 1
				\end{pmatrix},
				\qquad
				&
				A_2
				& =
				\frac{2}{3}\, 
				\begin{pmatrix}
					1 & 0\\ -1 & 1
				\end{pmatrix},
				\qquad
				&
				A_3
				& =
				\tfrac{2}{3}\, 
				\begin{pmatrix}
					2 & 0\\ 0 & \nicefrac{1}{2}
				\end{pmatrix}.
			\end{aligned}
		\end{split}
	\end{equation}
	\Cref{fig:transport_map_to_Gaussian_mixture_grid_and_pulled_back_function} illustrates the transport map established in \Cref{thm:TransportMixtures} to the grid constructed in \Cref{example:TransportMapUniformToNormal}, while \Cref{fig:TransportingMCqMCsparseGridsToMixture} visualises the application of \Cref{alg:ODE_sampling_of_mixtures} to various types of \betterpoints $X_{1},\dots,X_{N}$ with distribution $\simpv$ (MC, \ac{QMC}, and \ac{SG}).
	Both figures show how the ODE \eqref{equ:Mixture_transport_ODE} transports the corresponding points from $\simpv$ to $\trv$, and the bottom line of \Cref{fig:transport_map_to_Gaussian_mixture_grid_and_pulled_back_function} additionally demonstrates how the integrand $f$ is `pulled back' by $T$ to the function $f\circ T$.
	Further, in order to compare the performance of \ac{MC}, \ac{TQMC}, and \ac{TSG}, \Cref{fig:Three_Gaussians_Convergence_Plots} shows convergence plots for the quadrature problem $\bE_{Y\sim \trv}[f(Y)]$ with three different integrands $f$.
\end{example}

\begin{remark}
	\label{rem:pulled_back_function_almost_discontinuous}
	The transport map $T = \Phi_{1}$ constructed in \Cref{thm:TransportMixtures} can have some undesirable features.
	In particular, note how the two neighbouring grid points marked in green in \Cref{fig:transport_map_to_Gaussian_mixture_grid_and_pulled_back_function} are mapped by $T$ to rather distant points.
	As a consequence, the pulled back function $f \circ T$ is almost discontinuous in the corresponding region, which can cause challenges for certain quadrature rules.
	Such effects are inevitable:
	The reader may think of a transport map from a two-dimensional Gaussian distribution to an equal mixture of two Gaussians with far apart centers --- clearly, neighbouring points have to be `torn apart' along some line.
	In certain cases, adaptive schemes can provide a remedy for this issue.
	Note that our transport map relies on the solution of an ODE with \emph{analytic} right-hand side.
	Hence, no approximation of intermediate densities or other forms of particle interactions, as is typical for other methods, are necessary.
	Therefore, our approach can be trivially extended to such adaptive schemes, e.g.\ adaptive sparse grids.
\end{remark}

\begin{figure}[t]
	\centering         	
	\begin{subfigure}[b]{0.32\textwidth}
		\centering
		\includegraphics[width=\textwidth]{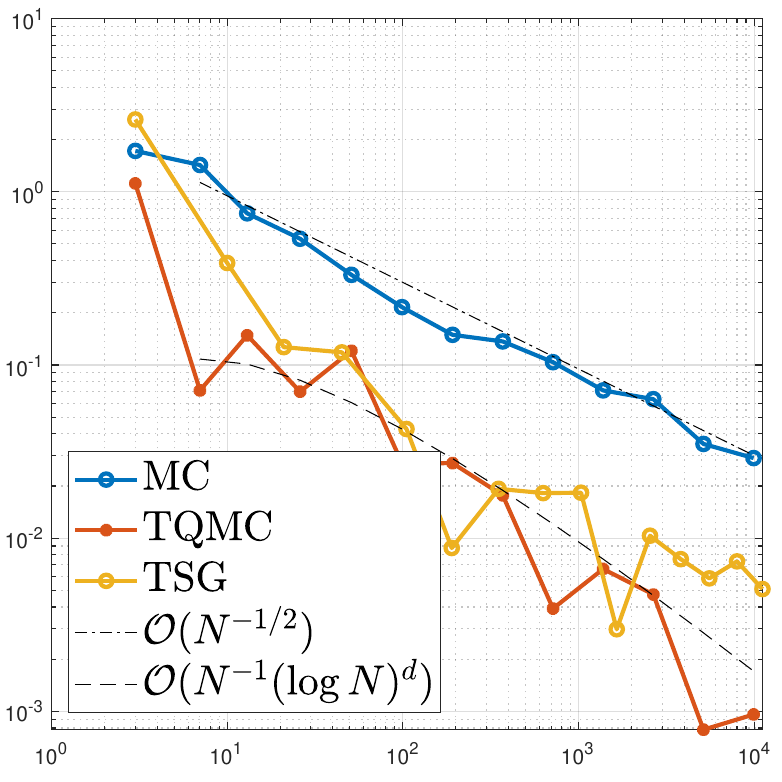}
	\end{subfigure}
	\hfill
	\begin{subfigure}[b]{0.32\textwidth}
		\centering
		\includegraphics[width=\textwidth]{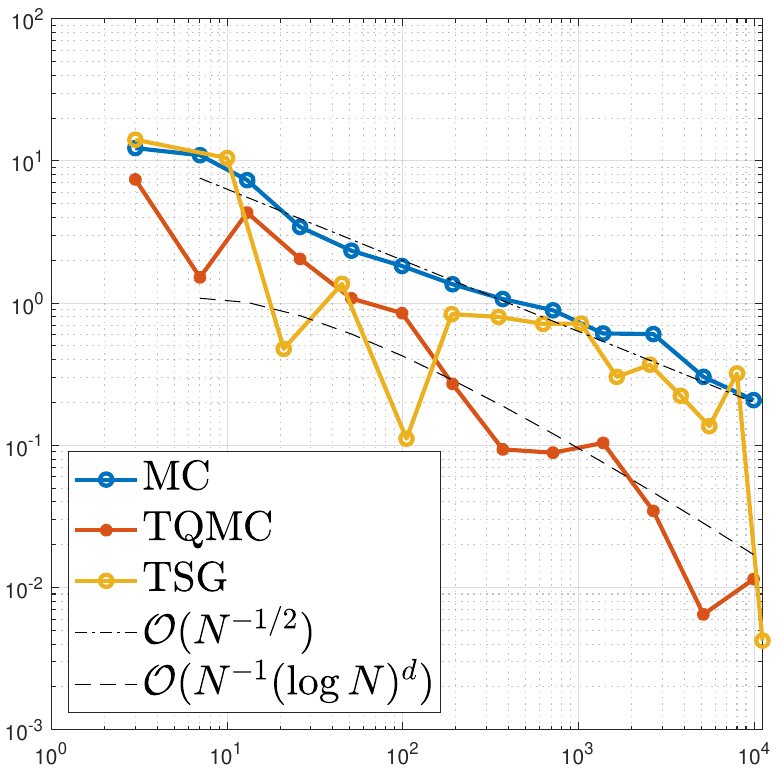}
	\end{subfigure}
	\hfill
	\begin{subfigure}[b]{0.32\textwidth}
		\centering
		\includegraphics[width=\textwidth]{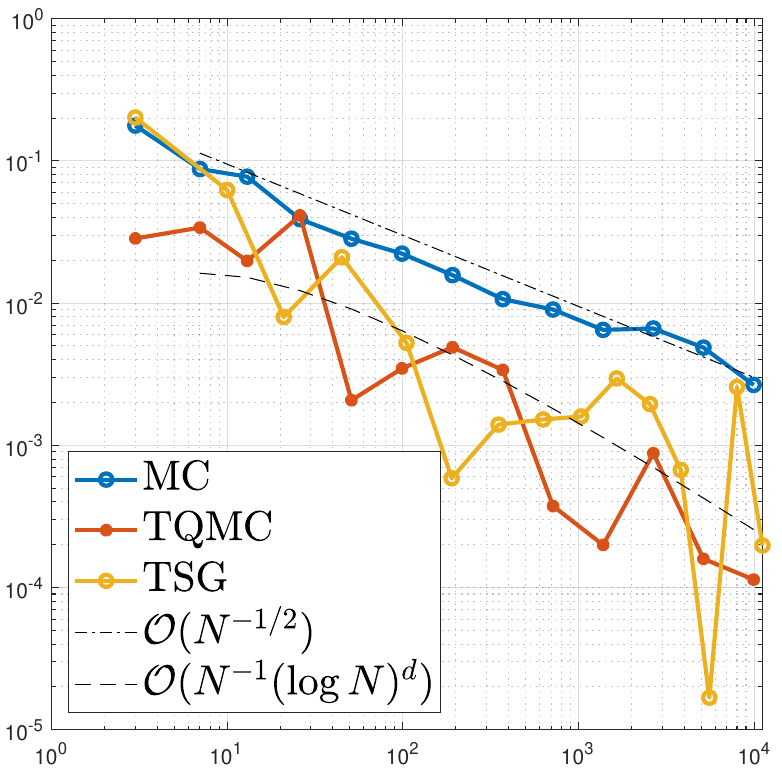}
	\end{subfigure}
	\caption{Convergence plots for the expected value $\bE_{\trv}[f]$ from \Cref{example:Transport_mixture_three_Gaussians} with target distribution $\trv$ given by the Gaussian mixture defined in \eqref{equ:Example_mixture_three_Gaussians} for three different integrands $f \in \{ f_{2}, f_{4}, f_{9} \}$ defined in \Cref{table:list_of_integrands}.
		MC is based on direct sampling from $\trv$ (\Cref{alg:direct_sampling_of_mixtures}), while TQMC and TSG rely on the ODE transport introduced in \Cref{thm:TransportMixtures} from a standard normal distribution $\simpv = \cN(0,\Id_2)$ to $\trv$ (\Cref{alg:ODE_sampling_of_mixtures}).
		To match the initial distribution, the initial points rely on the 2-dimensional Halton sequence transformed by $T^{\cN}$ for TQMC and on a sparse grids construction based on Gauss--Leja points in the case of TSG.		
		Clearly, transported \betterpoints outperform random samples in these examples.
	}
	\label{fig:Three_Gaussians_Convergence_Plots}
\end{figure}

\subsection{Mixture Distributions with Negative Weights}
\label{section:MixturesWithNegativeWeights}

In certain situations, probability densities are approximated by mixture distributions with weights $w_{j}$ that are not necessarily positive.
Such approximations arise e.g.\ in the context of reproducing kernel Hilbert spaces, the conditional density operator of \citep{schuster2020kcdo} being one example.
Note that the proof of \Cref{thm:TransportMixtures} remains legitimate for arbitrary weights as long as we can guarantee $\rho_t$ to be positive for all $t\in[0,1]$, which cannot be expected in general.
However, it does hold in the case of Gaussian densities
\[
	G[m,C](x)
	\defeq
	\frac{|\det C|^{-1/2}}{(2\pi)^{d/2}}
	\exp\left( -\tfrac{1}{2} (x-m)^\top C^{-1} (x-m) \right),
\]
if we do not allow scaling, i.e.\ $A_j = \Id_d$ for all $j=1,\dots,J$:

\begin{proposition}
	\label{prop:IntermediateDensitiesPositiveForGaussian}
	Let $\simpd = G[0,C]$ be a centered Gaussian density on $\bR^d$ with covariance matrix $C$ and let $\trd$ be given by \eqref{equ:MistureDensityScaled} with $A_j = \Id_d$ for all $j=1,\dots,J$.
	If $\trd $ is strictly positive, then the densities
	\[
		\rho_t = \sum_{j=1}^{J}w_j\rho_{j,t},
		\qquad
		\rho_{j,t} \defeq \simpd(x-ta_j),
		\qquad
		t\in[0,1],
	\]
	are strictly positive for all $t\in[0,1]$.
\end{proposition}

\begin{proof}
	If $\trd $ is positive, then so are the densities
	\[
		g_t(x)
		\defeq
		t^{-d}\trd  \left(t^{-1}x\right)
		=
		\sum_{j=1}^{J} w_j G[ta_j,t^2 C](x),
		\qquad
		t\in[0,1].
	\]
	Hence,
	\[
		\rho_t
		=
		\sum_{j=1}^{J} w_j \rho_{j,t}
		=
		\sum_{j=1}^{J} w_j G[t a_j,C]
		=
		g_t\ast G[0,(1-t^2)C]
	\]
	is also positive for each $t\in[0,1]$.
\end{proof}

\begin{corollary}
	\label{corollary:Gaussian_mixtures_with_negative_weights}
	If $\simpd$ is a centered Gaussian density as in \Cref{prop:IntermediateDensitiesPositiveForGaussian}, then the results of \Cref{thm:TransportMixtures} hold without the assumption $w_j\ge 0,\, j=1,\dots,J$, as long as $\trd$ is strictly positive and $A_j = \Id_d$, $j=1,\dots,J$ (no scaling).
\end{corollary}
	
\section{Numerical Experiments}
\label{section:NumericalExperiments}

\subsection{Extensive study with an adjustable mixture construction and several integrands}

In \Cref{example:Transport_mixture_three_Gaussians} and \Cref{fig:transport_map_to_Gaussian_mixture_grid_and_pulled_back_function,fig:TransportingMCqMCsparseGridsToMixture,fig:Three_Gaussians_Convergence_Plots} we have analysed a specific mixture of three Gaussian distributions.
In order to perform a more extensive and fair investigation of our \ac{TQMC} method for various dimensions $d$, numbers $J$ of mixture components, and integrands $f$, we deploy the following construction using a Gaussian mixture given by \eqref{equ:MistureDensityScaled} with $\simpv \sim \cN(0,\Id_d)$:
\begin{enumerate}[label=(\roman*)]
	\item \label{item:choice_mixture_centers}
	The centers $a_{j} \iidsim \cN(0,\Id_{d})$ are independent standard normal random variables.
	\item \label{item:choice_mixture_scaling_matrices}
	The scaling matrices $A_{j}$ are the Cholesky factors of covariance matrices $C_{j} = d \, \tilde{C}_{j}$, where $\tilde{C}_{j} \iidsim W_{d}(\Sigma,\nu)$ and $W_{d}(\Sigma,\nu)$ denotes the Wishart distribution with $\nu = d+4$ degrees of freedom with covariance matrix $\Sigma = \nu^{-1} \Id_{d}$.
	\item \label{item:choice_mixture_weights}
	We use equal weights $w_{j} = J^{-1}$.
	\item \label{item:description_integrand_list}
	We consider nine different integrands $f$ which were kindly provided by John Burkardt on his website: \url{https://people.math.sc.edu/Burkardt/f\_src/test\_nint/test\_nint.html}.
	To cover a broad range of integrands we chose the functions with the numbers 16, 17, 2 ,6, 7, 28, 30, 9, 27, and 18, listed in \Cref{table:list_of_integrands} (after relabelling).
\end{enumerate}
We compared our \ac{TQMC} method (\Cref{alg:ODE_sampling_of_mixtures}) with direct MC sampling (\Cref{alg:direct_sampling_of_mixtures}), and the results are summarised in \Cref{table:list_of_integrands}.

\Cref{fig:MC_QMC_comparison_for_various_dimensions_and_component_numbers} illustrates these results as convergence plots over the number $N$ of quadrature points for dimension $d \in \{ 2, 5, 20, 50 \}$ and $J \in \{ 2, 5, 20 \}$ mixture components for one specific integrand, namely $f(y) = f_{9}(y) = \cos\big( 0.3 + d^{-1}\sum_{j=1}^{d} y_{j} \big)$.

In addition, we have investigated how a CQMC strategy addressed in \Cref{remark:QMC_points_from_each_component_separately} and \Cref{alg:componentwise_transport_sampling_of_mixtures} performs as the number $J \in [1,2^{11}]$ of mixture components increases (for a fixed number $N=2^{12}$ of evaluation points), and illustrate the results in \Cref{fig:MC_QMC_error_over_J_dim_2}.
As expected, if $J$ is large, e.g.\ being of the order of the affordable number $N$ of evaluation points, this is no longer a feasible strategy, even for equal weights $w_{j} = J^{-1}$, and is even outperformed by MC for large values of $J$.

\begin{table}[p]
	\centering	
	\caption{%
		List of integrands for which the performance of MC and TQMC was compared, cf.\ \ref{item:description_integrand_list} above.
		Here, $y^{\ast} \in \bR^{d}$ is given by $y_{j}^{\ast} = 1/2$, $j=1,\dots,d$, $B_{r} (x) = \set{ z\in\bR^{d} }{ \| z-y^{\ast} \|_{2} \leq r }$
		denotes the closed ball of radius $r$ centered at $y^{\ast}$ and $\mathds{1}_{A}$ is the indicator function of a subset $A \subseteq \bR^{d}$.
	}
	\label{table:list_of_integrands}
	\begin{tabular}{p{12.5em}<{\centering} p{12.5em}<{\centering} p{12.5em}<{\centering}}		
		\specialrule{1pt}{1\jot}{-3\jot}
		\small
		\begin{justify}
			TQMC outperforms MC in low dimensions; no comparison possible in higher dimensions due to lack of convergence of both methods.
		\end{justify}
		&
		\small
		\begin{justify}
			TQMC outperforms MC in low dimensions; in higher dimensions the regime of convergence is not reached, but TQMC still performs better or comparable to MC.
		\end{justify}
		&
		\begin{justify}
			TQMC outperforms MC throughout the considered dimensions.
		\end{justify}
		\\		
		\specialrule{1.2pt}{-2\jot}{2\jot}
		&
		&
		$f_{1}(y) 
		=
		\norm{y - y^{\ast}}_{1}$
		\\
		\grayrule
		&
		&
		$f_{2}(y) 
		=
		\norm{y - y^{\ast}}_{2}^{2}$
		\\
		\grayrule
		&
		$f_{3}(y) 
		=
		\big(\sum_{j=1}^{d} (2 y_{j} - 1)\big)^{4}$
		&
		\\
		\grayrule
		$f_{4}(y) 
		=
		\prod_{j=1}^{d} 2\, \abs{2y_{j}-1}$
		&		
		&
		$\tilde{f}_{4} = f_{4}^{1/d}$
		\\
		\grayrule		
		$f_{5}(y) 
		=
		\prod_{j=1}^{d} \tfrac{\pi}{2}\, \sin (\pi y_{j})$
		&
		$\tilde{f}_{5} = \abs{f_{5}}^{1/d} \sgn(f_{5})$
		&
		\\
		\grayrule
		\small
		$f_{6}(y) 
		=
		\prod_{j=1}^{d} \big( 1 + \abs{y_{j} - y_{j}^{\ast}}^{2}\big)^{-1}$
		&
		&
		$\tilde{f}_{6} 
		=
		f_{6}^{1/d}$
		\\
		\grayrule
		&
		&
		\small
		$f_{7}(y) 
		=
		\exp\Big( - d^{-2} \norm{y - y^{\ast}}_{2}^{2} \Big)$
		\\
		\grayrule
		&	
		&
		$f_{8}(y) 
		=
		\exp\big( d^{-1} \sum_{j=1}^{d} y_{j} \big)$
		\\
		\grayrule
		&
		\small
		$f_{9}(y) 
		=
		\cos\big( 0.3 + d^{-1}\sum_{j=1}^{d} y_{j} \big)$
		&		
		\\
		\grayrule
		$f_{10} 
		=
		\mathds{1}_{B_{1/2}(y^{\ast})}$
		&
		&
		$\tilde{f}_{10} 
		=
		\mathds{1}_{B_{d}(y^{\ast})}$
		\\
		\specialrule{1pt}{2\jot}{0\jot}
	\end{tabular}
\end{table}

\begin{figure}[p]
	\centering
	\begin{subfigure}[b]{0.19\textwidth}
		\centering
		\includegraphics[width=\textwidth]{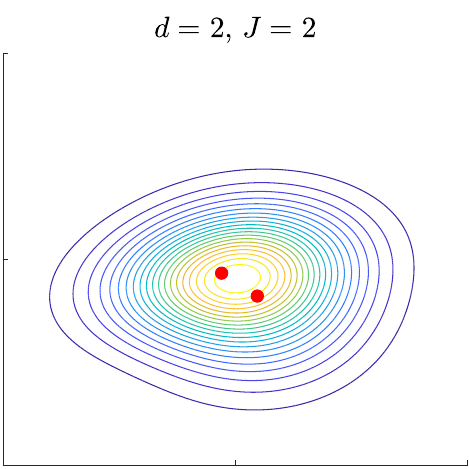}
	\end{subfigure}
	\hfill        	
	\begin{subfigure}[b]{0.209\textwidth}
		\centering
		\includegraphics[width=\textwidth]{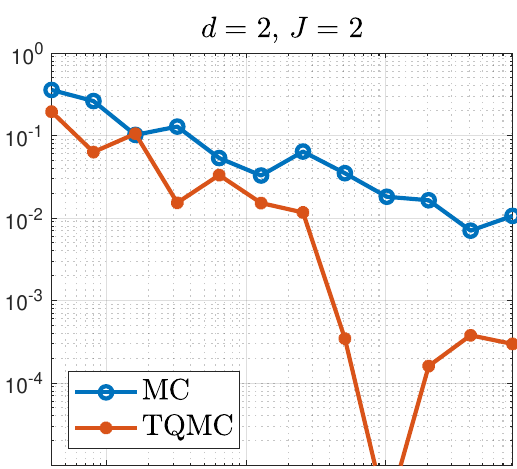}
	\end{subfigure}
	\hfill
	\begin{subfigure}[b]{0.19\textwidth}
		\centering
		\includegraphics[width=\textwidth]{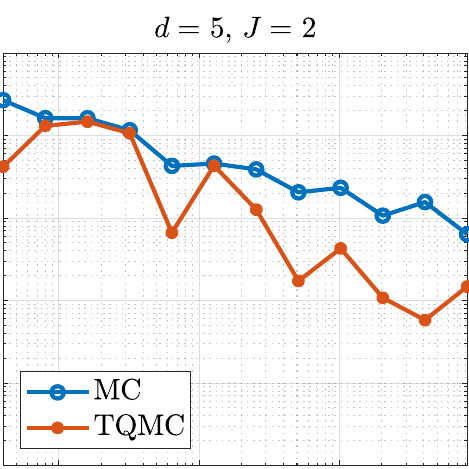}
	\end{subfigure}
	\hfill
	\begin{subfigure}[b]{0.19\textwidth}
		\centering
		\includegraphics[width=\textwidth]{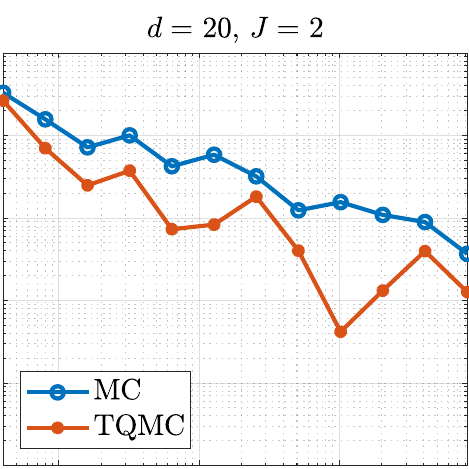}
	\end{subfigure}
	\hfill
	\begin{subfigure}[b]{0.19\textwidth}
		\centering
		\includegraphics[width=\textwidth]{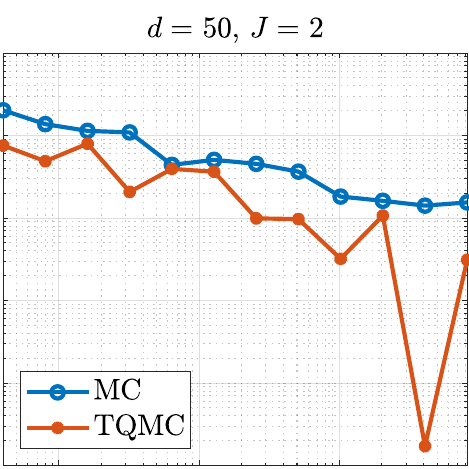}
	\end{subfigure}
	\vfill
	\begin{subfigure}[b]{0.19\textwidth}
		\centering
		\includegraphics[width=\textwidth]{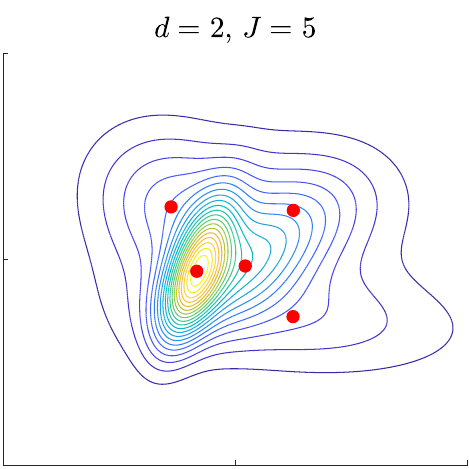}
	\end{subfigure}
	\hfill  
	\begin{subfigure}[b]{0.209\textwidth}
		\centering
		\includegraphics[width=\textwidth]{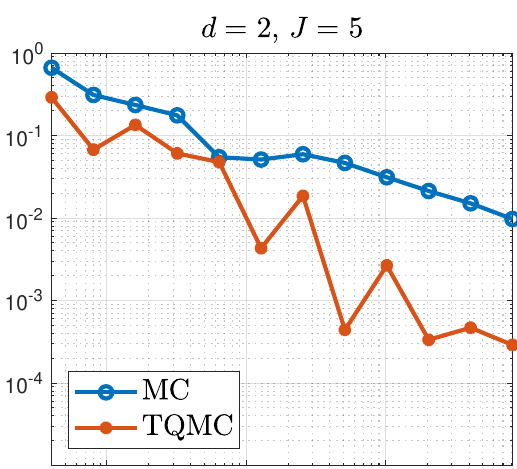}
	\end{subfigure}
	\hfill
	\begin{subfigure}[b]{0.19\textwidth}
		\centering
		\includegraphics[width=\textwidth]{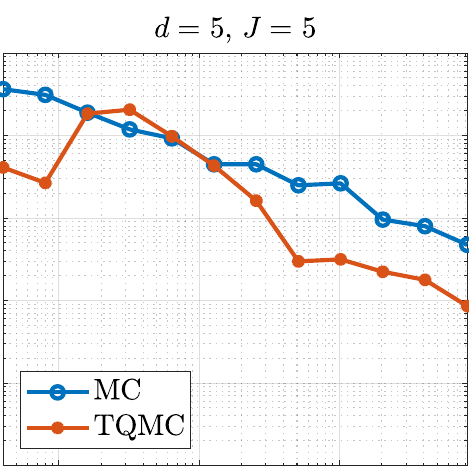}
	\end{subfigure}
	\hfill
	\begin{subfigure}[b]{0.19\textwidth}
		\centering
		\includegraphics[width=\textwidth]{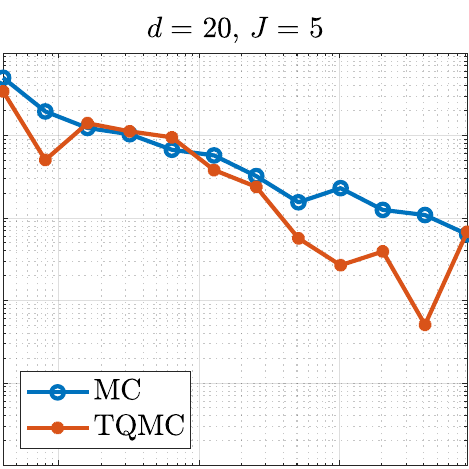}
	\end{subfigure}
	\hfill
	\begin{subfigure}[b]{0.19\textwidth}
		\centering
		\includegraphics[width=\textwidth]{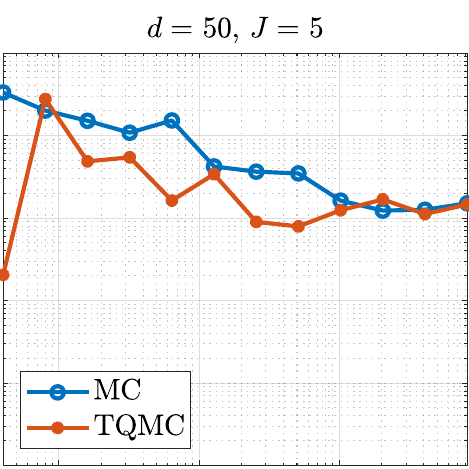}
	\end{subfigure}
	\vfill
	\begin{subfigure}[b]{0.19\textwidth}
		\centering
		\includegraphics[width=\textwidth]{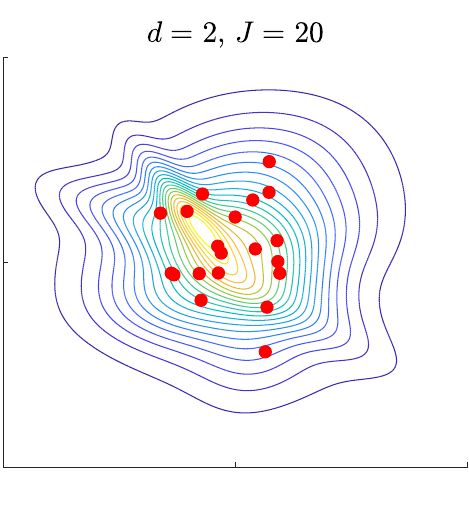}
	\end{subfigure}
	\hfill  
	\begin{subfigure}[b]{0.209\textwidth}
		\centering
		\includegraphics[width=\textwidth]{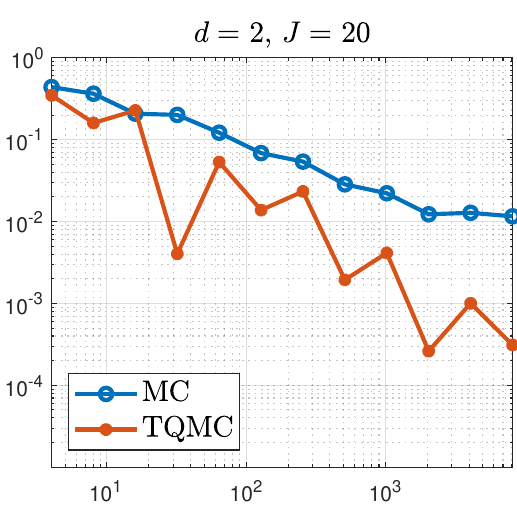}
	\end{subfigure}
	\hfill
	\begin{subfigure}[b]{0.19\textwidth}
		\centering
		\includegraphics[width=\textwidth]{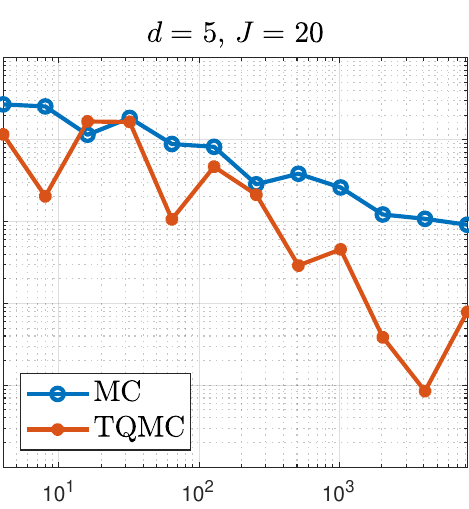}
	\end{subfigure}
	\hfill
	\begin{subfigure}[b]{0.19\textwidth}
		\centering
		\includegraphics[width=\textwidth]{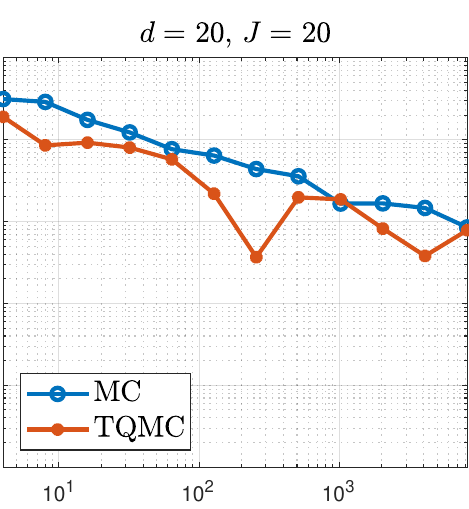}
	\end{subfigure}
	\hfill
	\begin{subfigure}[b]{0.19\textwidth}
		\centering
		\includegraphics[width=\textwidth]{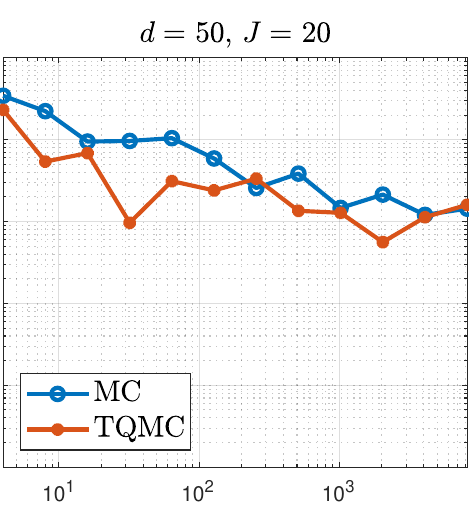}
	\end{subfigure}	
	\caption{Comparison of MC and TQMC in terms of the integration error plotted over the number $N\in [2^{2},2^{13}]$ of sample points for the function $f(y) = f_{9}(y) = \cos\big( 0.3 + d^{-1}\sum_{j=1}^{d} y_{j} \big)$ in different dimensions $d = 2,5,20,50$ and for various numbers $J=2,5,20$ of mixture components;
	see \ref{item:choice_mixture_centers}---\ref{item:choice_mixture_weights} on page \pageref{item:choice_mixture_centers}.
		For dimension $d=2$ a plot of the corresponding mixture density is added in the left panel.
	}
	\label{fig:MC_QMC_comparison_for_various_dimensions_and_component_numbers}
\end{figure}

\begin{figure}[p]
	\centering         	
	\includegraphics[width=0.5\textwidth]{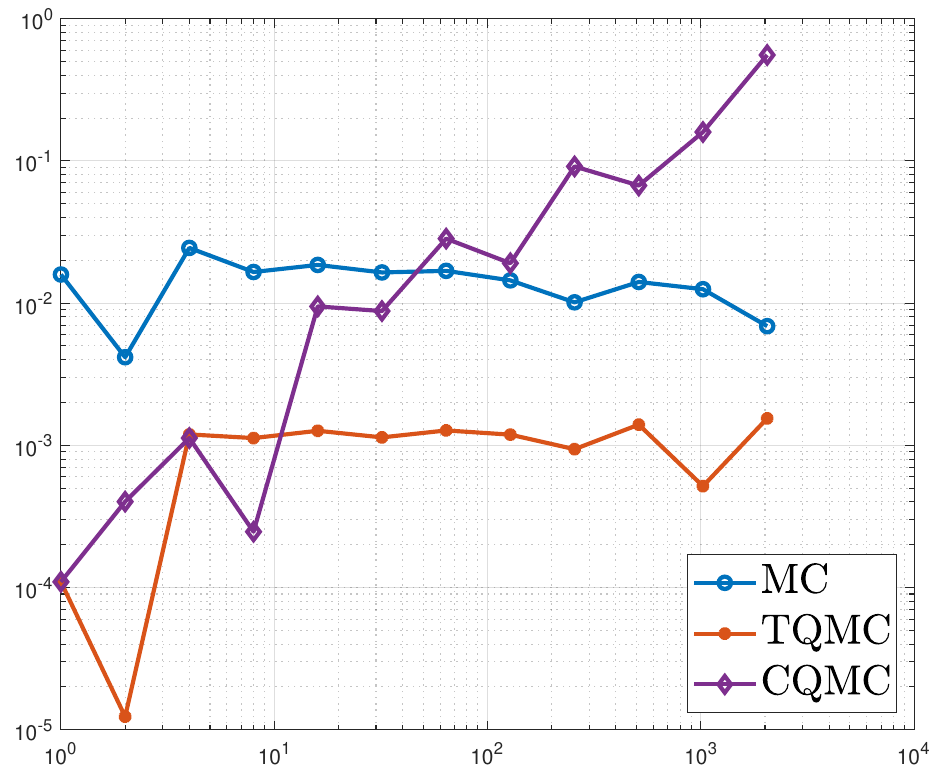}
	\caption{
		CQMC described in \Cref{remark:QMC_points_from_each_component_separately} (\Cref{alg:componentwise_transport_sampling_of_mixtures}) for a varying number $J \in [1,2^{11}]$ of mixture components with a fixed number of evaluation points $N = 2^{12}$, compared to MC sampling (\Cref{alg:direct_sampling_of_mixtures}) and our \ac{TQMC} approach (\Cref{alg:ODE_sampling_of_mixtures}).
		Clearly, the componentwise procedure is not meaningful for large values of $J$, while the other two are barely affected by varying $J$.
	}
	\label{fig:MC_QMC_error_over_J_dim_2}
\end{figure}

\subsection{Transported \ac{QMC} Points within \ac{LAIS}}
\label{section:Lais_with_QMC}

As mentioned in the introduction, the transport to mixture distributions should not be viewed as a toy example since it can be combined with various state-of-the-art importance sampling algorithms.
In this section, we demonstrate its practical importance by combining TQMC with \acl{LAIS} (\acs{LAIS}; \citealp{martino2017layered}).
\ac{LAIS} and its various versions \citep{martino2017layered,bugallo2017adaptive,martino2017anti} are powerful and up-to-date \ac{MC} methods that build an approximation to the target distribution $\trv$ by a mixture distribution in a first step (`upper layer'), typically by running one or more Markov chains, and using it as an importance sampling distribution (\citealp[Section~5.7]{rubinstein2016simulation}; \citealp[Section~3.3]{robert2004monte}) in a second step (`lower layer').

Leaving the upper layer untouched, we can modify the sampling of the lower layer by replacing the random samples with transported \betterpoints.
To ensure a comparison fair to \ac{LAIS}, we illustrate the benefits of this modification on a multimodal two-dimensional target distribution, visualised in \Cref{fig:LAIS_Density_with_points}, that corresponds exactly to the one used by \citet{martino2017layered}, as do the Markov chains used in the upper layer and the importance sampling distribution built from them.
Further, we run \ac{LAIS} using the `full deterministic mixture approach' \citep{martino2017layered} described below and from now on abbreviated by DM-\ac{LAIS}, which gives the best results in terms of the variance of the estimator at the price of a higher computational cost, thus getting the best out of \ac{LAIS}.

\begin{figure}[t]
	\centering         	
	\begin{subfigure}[b]{0.48\textwidth}
		\centering
		\includegraphics[width=\textwidth]{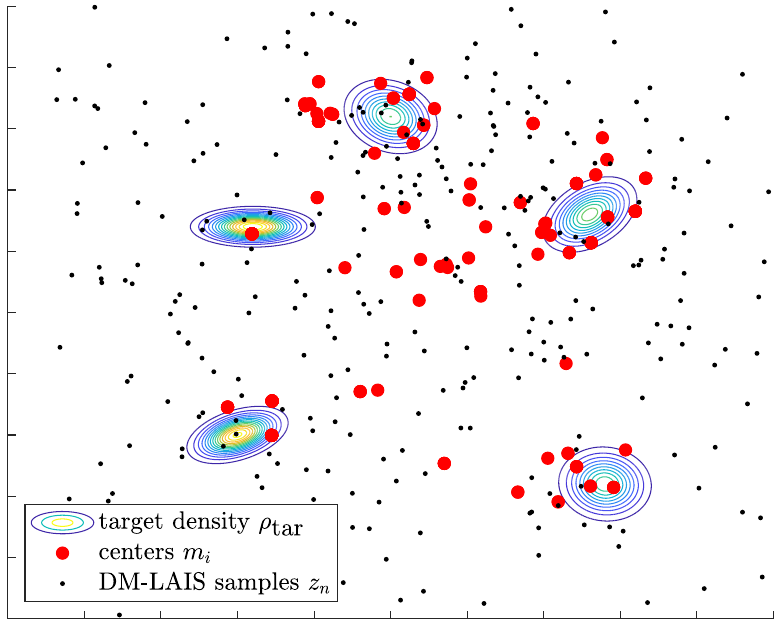}
	\end{subfigure}
	\hfill
	\begin{subfigure}[b]{0.48\textwidth}
		\centering
		\includegraphics[width=\textwidth]{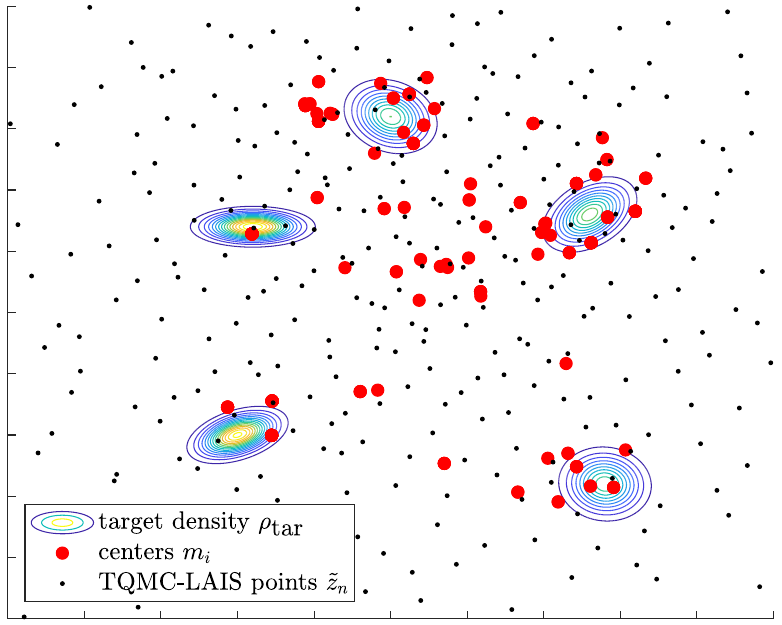}
	\end{subfigure}
	\caption{Target density used by \citet{martino2017layered}, together with the centers $m_{i}$ and DM-LAIS samples $z_{n}$ (left) compared to \ac{TQMC}-\ac{LAIS} points $\tilde{z}_{n}$ (right).
		As expected, the \ac{TQMC}-\ac{LAIS} points are more evenly spaced, resulting in a higher convergence rate of the corresponding Monte Carlo estimator, as shown in \Cref{fig:LAIS_convergence}.
	}
	\label{fig:LAIS_Density_with_points}
\end{figure}

DM-\ac{LAIS} is based on the samples $m_{1},\dots,m_{CT}$ produced by $C\in\bN$ parallel Metropolis--Hastings chains that are run for $T\in\bN$ time steps (upper layer).
These samples are then used to build the importance sampling density $\hat\rho = (CT)^{-1} \sum_{i=1}^{CT} q(\quark | m_{i})$, where $q(\quark | m_{i})$ are easily sampled probability density functions (oftentimes these are chosen to coincide with the proposal densities \citep{schuster2021MCIS}; the reader may think of Gaussian kernels centered at $m_{i}$).
Hereafter, $M \in \bN$ \ac{IID} samples are produced from each of the components $q(\quark | m_{i})$, resulting in a so-called \emph{stratified} sampling $(z_{1},\dots,z_{CTM})$ from $\hat\rho$ \citep[Chapter~5.5]{rubinstein2016simulation}, which, compared to independent samples from $\hat\rho$ produced by \Cref{alg:direct_sampling_of_mixtures}, result in a smaller variance of the \ac{MC} estimator \citep[Proposition~5.5.1]{rubinstein2016simulation}.
These samples are then used in a standard (self-normalised) importance sampling scheme to approximate the expected value $\bE_{\trv}[f]$ of some function $f\in L^{1}(\trv)$ by
\[
	\hat{S}^{\textup{snIS}}[f]
	=
	\frac{\sum_{n=1}^{CTM} \omega_{n} f(z_{n})}{\sum_{n=1}^{CTM} \omega_{n}},
	\qquad
	\omega_{n}
	=
	\frac{\trd(z_{n})}{\hat\rho(z_{n})}.
\]

We suggest to replace the samples $z_{n}$ by \ac{TQMC} points $\tilde{z}_{n}$ transported to the mixture distribution $\hat\rho$ as in \Cref{thm:TransportMixtures} (\Cref{alg:ODE_sampling_of_mixtures}) and term the resulting algorithm \ac{TQMC}-\ac{LAIS}.
Naturally, for this to work we require $\hat\rho$ to be of the form \Cref{equ:MistureDensityScaled}, i.e.\ the densities $q(\quark | m_{i})$ have to be scaled and shifted copies of $\simpd$.
While any other \betterpoints can be used, we employed \ac{QMC} points as a proof of concept.
For practical applications with extremely long Markov chains, it is meaningful to \emph{thin} the chains \citep{owen2017thinning,riabiz2022thinning} in order to keep the number of mixtures (and thereby the computational cost) manageable.

We make two comparisons between DM-\ac{LAIS} and \ac{TQMC}-\ac{LAIS}.
First, we fix the number of chains $C = 10$ and number of time steps $T = 20$ and vary the number of samples per mixture component $M$.
Second, we fix the number of chains $C = 10$ and number of samples per mixture component $M=100$ and vary the number of time steps $T$.
In both cases we plot the approximation error $\| \hat{S}^{\textup{snIS}}[f] - \bE_{\trv}[f]\|$ for $f(x) = x$ over the total number of samples $N = CTM$, illustrated in \Cref{fig:LAIS_convergence}.
We observe a higher convergence rate of \ac{TQMC}-\ac{LAIS} compared to DM-\ac{LAIS}.

\begin{figure}[t]
	\centering         	
	\begin{subfigure}[b]{0.48\textwidth}
		\centering
		\includegraphics[width=\textwidth]{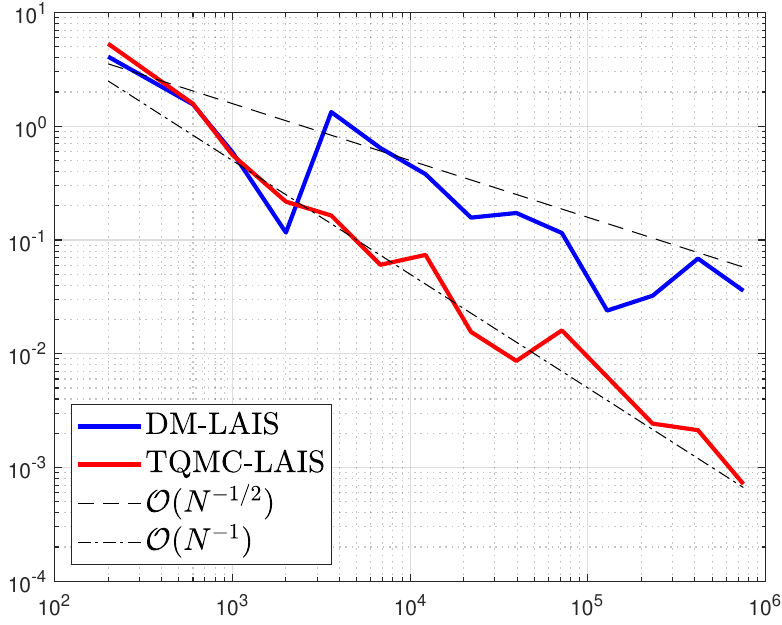}
	\end{subfigure}
	\hfill
	\begin{subfigure}[b]{0.48\textwidth}
		\centering
		\includegraphics[width=\textwidth]{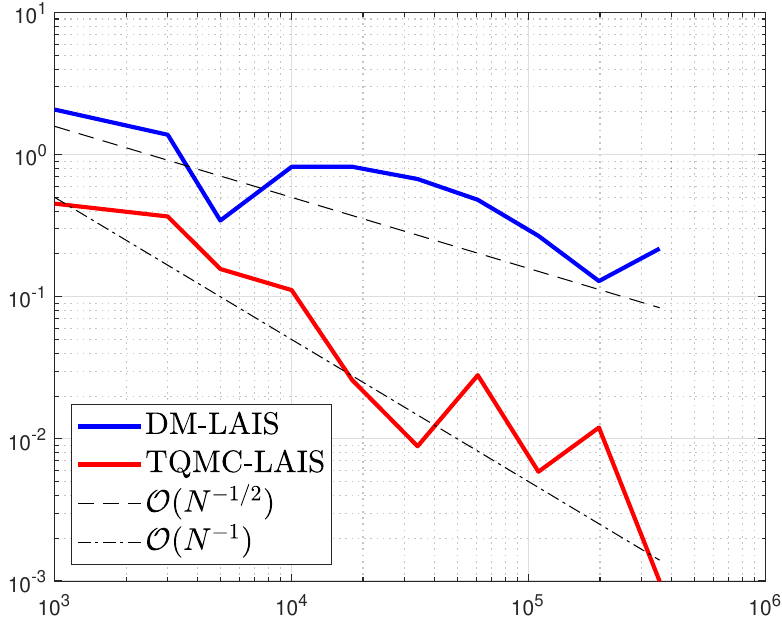}
	\end{subfigure}
	\caption{Monte Carlo quadrature error over the number of samples for \ac{TQMC}-\ac{LAIS} compared to DM-\ac{LAIS}.
		As expected, the \ac{TQMC} points show a higher convergence rate than the random DM-\ac{LAIS} samples.
		\textit{Left:} number of chains $C = 10$ and number of time steps $T = 20$ fixed, number of samples per mixture component $M$ varied.
		\textit{Left:} number of chains $C = 10$ and number of samples per mixture component $M=100$ fixed, number of time steps $T$ varied.
	}
	\label{fig:LAIS_convergence}
\end{figure}
	
\section{Implementation Details}
\label{section:ImplementationDetails}

In all our numerical experiments, QMC points relied on the 2-dimensional Halton sequence (transformed by $T^{\cN}$ to match the initial Gaussian distribution) using the following \textsc{Matlab} commands:
\begin{greennote}
\begin{lstlisting}
  p = haltonset(dim,'Skip',1e3,'Leap',1e2);
  p = scramble(p,'RR2');
\end{lstlisting}
\end{greennote}
\noindent
For sparse grids we used the \textsc{sparse grids Matlab kit} developed by Chiara Piazzola, Lorenzo Tamellini and their coworkers \citep{piazzola2023SGK}, available at
\begin{center}
	\href{https://sites.google.com/view/sparse-grids-kit}{https://sites.google.com/view/sparse-grids-kit}.
\end{center}
For the specific sparse grids constructions based on Gauss--Hermite and Gauss--Leja points the following \textsc{Matlab} commands were used, respectively:
\begin{greennote}
\begin{lstlisting}
  knots = @(n) knots_gaussian(n,0,1);
  knots = @(n) knots_normal_leja(n,0,1,'sym_line');
\end{lstlisting}
\end{greennote}
\noindent
followed by the standard commands to obtain the reduced sparse grids and weights:
\begin{greennote}
\begin{lstlisting}
  S       = smolyak_grid(2,level,knots,@lev2knots_lin);
  Sr      = reduce_sparse_grid(S);
  samp_0  = Sr.knots;
  weights = Sr.weights;    
\end{lstlisting}
\end{greennote}
\noindent
For the experiments with \ac{LAIS} in \Cref{section:Lais_with_QMC} we used the \textsc{Matlab} code kindly provided by Luca Martino on his webpage \href{http://www.lucamartino.altervista.org/code.html}{www.lucamartino.altervista.org/code.html}.
	\section{Conclusion and Outlook}
\label{section:Conclusion}

In this paper, we have derived an explicit transport map $T\colon \bR^{d} \to \bR^{d}$ from some probability distribution $\simpv$ on $\bR^{d}$ to a mixture $\trv$ of scaled and shifted versions of $\simpv$ (\Cref{thm:TransportMixtures} and \Cref{alg:ODE_sampling_of_mixtures}).
By `explicit' we mean that it is the solution of an ODE with \emph{analytic} right-hand side.
This transport map can be used to transport \betterpoints and thereby quadrature rules previously established for $\simpv$ to such a mixture, allowing for super-root-$N$ convergence of the corresponding empirical mean towards the analytic mean $\bE_{\trv}[f] \defeq \int f(y)\, \rd \trv(y)$ for certain integrands $f$.

While most distributions in practice, such as Bayesian posteriors, are not mixture distributions, many methodologies approximate the distribution of interest by such a mixture, e.g.\ \ac{LAIS} as well as certain versions of variational inference, such as variational boosting.
Hence, combining the above approach with such methods by using an importance sampling estimator allows to approximate expected values with respect to more complicated target distributions with improved convergence rates.

Since the right-hand side of the constructed ODE is explicit, $T$ can be evaluated for each point separately, without any particle interactions, that are typical for most particle flow methods.
Hence, our approach is trivially parallelisable and can be easily extended to adaptive schemes, such as adaptive sparse grids.

We demonstrated the advantages of our approach for a simple mixture of three Gaussian distributions in dimension $d = 2$ and performed an extensive study over different dimensions $d = 2,5,20,50$ and for various numbers of mixture components.
In addition, we showed how LAIS can benefit from using transported \betterpoints in place of random samples in the sampling step ('lower layer').

In order to establish the theory for transported \betterpoints, one crucial step remains as an open problem, namely characterisation of the class $\cC_T$ of admissible integrands $f$.
As argued in \Cref{remark:transformed_class_of_functions}, the corresponding detailed analysis goes beyond the scope of this paper.
	
	\appendix
	\section{Supporting Technical Results}
\label{section:Appendix}

\begin{lemma}
	\label{lemma:LinearChangeOfVariablesForDensities}
	Let $X\sim \bP_X$, where $\bP_X$ is a probability distribution with density $\rho_X\colon\bR^d\to\bR_{\ge 0}$, and $a\in\bR_{\ge 0}^d$, $A\in\GL(d,\bR)$, with $\GL(d,\bR)$ denoting the $d$-dimensional general linear group.
	Then $Y = AX+a$ has the probability density
	\[
		\rho_Y(y) = \abs{\det A}^{-1}\rho_X(A^{-1}(y-a)).
	\]
\end{lemma}

\begin{proof}
	This is simply the linear change of variables formula for probability densities.
\end{proof}

\begin{theorem}[{\citet[Proposition~8.1.8]{ambrosio2008gradient}}]
	\label{thm:continuity_equation_Ambrosio_version}
	Let $(\bP_{t})_{t \in [0,T]}$, $T > 0$, be a narrowly continuous family of Borel probability measures on $\bR^{d}$, $d\in\bN$, solving the continuity equation $\partial_t \bP_{t} = -\diver(v_{t} \bP_{t})$ on $(0,T) \times \bR^{d}$ with respect to a Borel vector field $v_{t}$ satisfying
	\begin{enumerate}[label = (\roman*)]
		\item
		\label{item:continuity_equation_technical_assumption_1}
		$\int_{0}^{T} \int_{\bR^{d}} \norm{v_{t}(x)} \, \mathrm d\bP_{t} \, \mathrm dt < \infty$
		\item 
		\label{item:continuity_equation_technical_assumption_2}
		$\int_{0}^{T} \sup_{B} \norm{v_{t}} + \mathrm{Lip}(v_{t},B) \, \mathrm dt < \infty$
		for every compact $B \subseteq \bR^{d}$,
	\end{enumerate}
	where $\mathrm{Lip}(f,A)$ denotes a Lipschitz constant of the function or vector field $f$ in the set $A$.
	Then, for $\bP_{0}$-almost every $x\in\bR^{d}$, the characteristic system
	\[
		\tfrac{\partial}{\partial t} \Phi_{t}(x)
		=
		v_{t} (\Phi_{t}(x)),
		\qquad
		\Phi_{0}(x)
		=
		x,
	\]
	admits a globally-defined solution $(\Phi_{t})_{t\in [0,T]}$ and, for each $t\in [0,T]$, $\bP_{t} = (\Phi_{t})_{\#} \bP_{0}$.
\end{theorem}

\begin{proof}[Proof of \Cref{thm:TransportMixtures}]
	\citet[Lemma~8.1.2 and Remark~5.1.1]{ambrosio2008gradient} show that $(\bP_{t})_{t \in [0,1]}$ is a narrowly continuous family of Borel probability measures.
	We will now show that it solves the continuity equation.
	First note that, by a technical but straightforward calculation, each family $(\rho_{j,t})_{t \in [0,1]}$ solves the continuity equation
	\begin{equation}
		\label{equ:ContinuityEquationForEachJ}
		\partial_t\rho_{j,t} = -\diver(\rho_{j,t}v_{j,t}),
		\qquad
		\rho_{j,0} = \simpd,
		\qquad
		t \in [0,1],\ j=1,\dots,J.
	\end{equation}
	See \Cref{remark:TechnicalContinuityEquationMixtures} for an alternative justification of \eqref{equ:ContinuityEquationForEachJ}.
	Therefore, by linearity of differentiation, the family of probability densities $(\rho_t)_{t \in [0,1]}$ solves the continuity equation
	\[
		\partial_t\rho_t
		=
		\partial_t \Big(\sum_{j=1}^{J}w_j\rho_{j,t}\Big)
		=
		- \diver\Big(\sum_{j=1}^{J}w_j\rho_{j,t}v_{j,t}\Big)
		=
		- \diver(\rho_t v_t),
		\qquad
		\rho_{0} = \simpd,
		\qquad
		t \in [0,1].
	\]
	Now let us verify conditions \ref{item:continuity_equation_technical_assumption_1} and \ref{item:continuity_equation_technical_assumption_2} on the vector field $v_{t}$ from \Cref{thm:continuity_equation_Ambrosio_version}.
	Using the assumption on the first moment of $\simpv$ and the change of variables $z = A_{j,t}^{-1}(x-ta_j)$, we obtain
	\begin{align*}
		\int_{\bR^{d}} \norm{v_{t}(x)} \, \mathrm d\bP_{t}
		&\leq
		\sum_{j=1}^{J} w_j \int_{\bR^{d}} \rho_{j,t}(x) \, v_{j,t}(x)\, \mathrm dx
		\\
		&=
		\sum_{j=1}^{J} w_j \int_{\bR^{d}} \simpd(z) \, \norm{a_{j} + (A_{j}-\Id_{d}) z}\, \mathrm dz
		\\
		&\leq
		\sum_{j=1}^{J} w_j \big( \norm{a_{j}} + \norm{A_{j}-\Id_{d}} M \big).
	\end{align*}
	Since the bound on the right-hand side is independent of $t \in [0,1]$, this proves \ref{item:continuity_equation_technical_assumption_1}.
	In order to prove \ref{item:continuity_equation_technical_assumption_2}, fix a compact subset $B \subseteq \bR^{d}$ and consider the function
	\[
		V\colon [0,1] \times \bR^{d} \to \bR^{d},
		\qquad
		V(t,x) = v_{t}(x),
	\]
	which is continuously differentiable by construction.
	Hence, it attains its maximum on $[0,1]\times B$, i.e. $\hat{V} \defeq \sup_{[0,1]\times B} V < \infty$, and, by the mean value theorem, $\hat{L} \defeq \mathrm{Lip}(V,[0,1]\times B) < \infty$.
	It follows that $\int_{0}^{1} \sup_{B} \norm{v_{t}} + \mathrm{Lip}(v_{t},B) \, \mathrm dt \leq \hat{V}+\hat{L} < \infty$, proving \ref{item:continuity_equation_technical_assumption_2}.
	Since $\rho_{0} = \simpd$ and $\rho_{1} = \trd$, \Cref{thm:continuity_equation_Ambrosio_version} proves the claim.
\end{proof}

\begin{remark}
	\label{remark:TechnicalContinuityEquationMixtures}
	In place of the technical verification of \eqref{equ:ContinuityEquationForEachJ}, we can give the following much more intuitive justification:
	For fixed $j=1,\dots,J$ and $x_0\sim \simpv$, consider the ODE
	\begin{equation}
		\label{equ:ODEforEachJ}
		\dot{x}(t) = a_j + (A_j - \Id_d)x_0,
		\qquad
		x(0) = x_0,
	\end{equation}
	which has the straightforward solution
	\begin{equation}
		\label{equ:SolutionOfODEforEachJ}
		x(t)
		=
		x_0 + t \big( a_j + (A_j-\Id_d)x_0)
		=
		(tA_j + (1-t)\Id_d) x_0 + t a_j
		=
		A_{j,t} x_0 + t a_j.
	\end{equation}
	By \Cref{lemma:LinearChangeOfVariablesForDensities}, $x(t)$ has the probability density $\rho_{j,t}$.
	Note that the right-hand side of \eqref{equ:ODEforEachJ} agrees with $v_{j,t}$ if we rewrite $x_0$ in terms of $x(t)$ via \eqref{equ:SolutionOfODEforEachJ}. Hence, $(\rho_{j,t})_{t \in [0,1]}$ solves \eqref{equ:ContinuityEquationForEachJ}.
\end{remark}
	
	\section*{Acknowledgements}
	\addcontentsline{toc}{section}{Acknowledgements}
	
	IK and TJS have been supported in part by the Deut\-sche For\-schungs\-ge\-mein\-schaft (DFG) through through projects TrU-2 and EF1-10 of the Berlin Mathematics Research Centre MATH+ (EXC-2046/1, project \href{https://gepris.dfg.de/gepris/projekt/390685689}{390685689}).
	TJS has been further supported by the DFG through project \href{https://gepris.dfg.de/gepris/projekt/415980428}{415980428}.
	
	The authors thank Caroline Lasser, Claudia Schillings, Vesa Kaarnioja and Lorenzo Tamellini for helpful and collegial discussions.
	
	Some content of \Cref{section:TransportToMixtures} appeared in the PhD thesis of \citet{Klebanov2016ApproximationOP} but has not been previously submitted to any peer-reviewed journal.
	
	For the purpose of open access, the authors have applied a Creative Commons Attribution (CC BY) licence to any Author Accepted Manuscript version arising.
	
	\bibliographystyle{abbrvnat}
	\bibliography{myBibliography}
	\addcontentsline{toc}{section}{References}
	
\end{document}